\documentclass{amsart}

\usepackage[all]{xy}
\usepackage{amssymb,latexsym}
\usepackage{mltex}
\usepackage{amsmath}
\usepackage[colorlinks=true,linkcolor=blue,citecolor=blue]{hyperref}
 \usepackage{lscape}
\usepackage{enumerate}
\usepackage{amsthm}
\usepackage{hyperref}
\usepackage{multicol}

\newtheorem{thm}{Theorem}

\newtheorem{lem}{Lemma}[section]
\newtheorem{prop}[lem]{Proposition}
\newtheorem{rem}{Remark}

\newcommand{\C}{\mathbb{C}}

\newcommand{\HH}{\mathbb{H}}
\newcommand{\R}{\mathbb{R}}

\newcommand{\N}{\mathbb{N}}

\begin{document}
\title{Spinorial representation of submanifolds in $SL_n(\C)/SU(n)$}
\author{Pierre Bayard}
\email{bayard@ciencias.unam.mx}
\address{Facultad de Ciencias, Universidad Nacional Aut\'onoma de M\'exico
\\Av. Universidad 3000, Circuito Exterior S/N
\\Delegaci\'on Coyoac\'an, C.P. 04510, Ciudad Universitaria, CDMX, M\'exico}
\maketitle

\begin{abstract}
We give a spinorial representation of a submanifold of any dimension and co-dimension in a symmetric space $G/H,$ where $G$ is a complex semi-simple Lie group and $H$ is a compact real form of $G.$ This in particular includes $SL_n(\C)/SU(n),$ and extends the previously known spinorial representation of a surface in $\HH^3$ if $n=2.$ We also recover the Bryant representation of a surface with constant mean curvature 1 in $\HH^3$ and its generalization for a surface with holomorphic right Gauss map in $SL_n(\C)/SU(n).$ As a new application, we obtain a fundamental theorem for the submanifold theory in that spaces.
\end{abstract}

\date{}
\maketitle\pagenumbering{arabic}
\section{Introduction}
The Weierstrass representation formula permits to describe locally a minimal surface of $\R^3$ by means of two holomorphic functions. It is fundamental in the theory of minimal surfaces since it relates the theory to complex analysis and also allows the construction of many examples. This representation was extended for surfaces of arbitrary co-dimension, and also for surfaces with constant mean curvature in many other geometric contexts, for example for CMC-surfaces in $\R^3$ and $\R^4$ \cite{Ke,HO}, for surfaces with constant mean curvature 1 in $\HH^3$ \cite{Br} or recently for CMC-surfaces in 3-dimensional metric Lie groups \cite{MP}. There also exist other Weierstrass-type representation formulas for surfaces with constant Gauss curvature, for example for flat surfaces in $\mathbb{S}^3$ \cite{Bi}, in $\mathbb{H}^3$ \cite{GMM2}, or for flat surfaces with flat normal bundle in $\R^4$ \cite{dCD,DT,GM1}. These formulas also have their natural counterparts in other pseudo-riemannian space forms, as for maximal surfaces \cite{Kob}, timelike minimal surfaces \cite{Ma}, CMC-surfaces \cite{AN}, or surfaces with constant negative Gauss curvature \cite{GMM1} in 3-dimensional Minkowski space, for CMC-surfaces in de Sitter 3-space \cite{AA}, or for flat surfaces with flat normal bundle in $\R^{1,3}$ \cite{DT,GMM3} or $\R^{2,2}$ \cite{Pa}. All these representation formulas have strong analogies with the original Weierstrass formula. It seems in fact plausible that a general abstract formula do exist, giving rise to the various concrete formulas once the geometric context is specified. General representation formulas were indeed obtained for surfaces in $\R^3$ and $\R^4$ (see e.g. \cite{Ko1, Ko2, Ta1,Ta2,Fr} and the references therein), in $\mathbb{S}^3$ and $\HH^3$ \cite{Mo}, in 3-dimensional metric Lie groups \cite{BT}, in the Berger spheres \cite{LH1}, in 3-dimensional homogeneous spaces \cite{Ro}, in 4-dimensional space forms \cite{BLR1} and in some 3 and 4-dimensional Lorentzian space forms \cite{Va,La,LR,LH2,Bay,BP}; it appears that spinorial geometry provides the tools to write these general formulas in an efficient and elegant manner. We recently obtained with M.-A. Lawn, J. Roth and B. Zavala Jim\'enez general spinorial representation formulas for immersions in space forms \cite{BLR2} and in metric Lie groups \cite{BRZ}; the dimension and the co-dimension are arbitrary. 

The purpose of the paper is to give a general spinorial representation of a submanifold of any dimension and co-dimension in a symmetric space $G/H,$ where $G$ is a complex semi-simple Lie group and $H$ is a compact real form of $G$ (the group $G$ will be moreover assumed to be simply connected and the subgroup $H$ connected). This includes $SL_n(\C)/SU(n), n\geq 2$ and thus also $\HH^3=SL_2(\C)/SU(2)$ as special cases. In these spaces important Weierstrass-type formulas are known: the Bryant representation for surfaces with constant mean curvature 1 in $\HH^3,$ and its generalization by Kokubu, Takahashi, Umehara and Yamada for surfaces with holomorphic right Gauss map in $SL_n(\C)/SU(n),$ $n\geq 3$ \cite{KTUY}; we recover these representation formulas from the general abstract formula. Let us stress here that the spinor bundle used in the paper has a higher rank than the usual spinor bundle, as in the papers \cite{BLR2,BRZ}: it is associated to the left multiplication of the spin group on the Clifford algebra, rather than to the usual irreducible spinor representation; it is thus a sum of copies of the usual spinor bundle, whose number of factors increases exponentially with the dimension of $G/H.$ However, it is maybe not possible to obtain a general representation formula using bundles of lower dimension. 

As a first application of our general spinorial representation formula we obtain a fundamental theorem for the submanifold theory in $G/H$ (Theorem \ref{thm fundamental} in Section \ref{section fundamental equations}). We then recover as a special case the spinorial representation of a general surface in $\HH^3$ given by Morel in \cite{Mo} (Section \ref{section H3 morel}), and also the Bryant representation of a surface with constant mean curvature 1 in $\HH^3$ (Section \ref{section H3 CMC1}). We finally recover the generalized Bryant representation of a surface with holomorphic right Gauss map in $SL_n(\C)/SU(n)$ given by Kokubu, Takahashi, Umehara and Yamada. This gives a new spinorial interpretation of these representations  (see also \cite{GM2,MN1}), and also gives a relation between the abstract spinorial representation given by Morel and the Bryant representation.

It should be possible to adapt the basic constructions in the paper in order to obtain spinorial representations of submanifolds in other symmetric or homogeneous spaces. We hope to come back soon to that question. We thank a referee for suggesting us the use of $Spin^c$ structures instead of $Spin$ structures, which could be more natural in some contexts (see e.g. \cite{NR} and the references therein): this could be also a theme of future research.

The paper is organized as follows. Section \ref{section preliminaries} is devoted to preliminaries on the symmetric Lie algebra, the canonical connection and the Maurer-Cartan form on $G/H.$ We describe in Section \ref{section cartan embedding} the Cartan embedding of $G/H$ into the spin group $Spin(\mathfrak{g})$ and in Section \ref{section spinor bundle GH} the spinor bundle on $G/H.$ We construct the abstract spinor bundle and the various objects naturally defined on it in Section \ref{section abstract setting}. We state the main result in Section \ref{section main result} (Theorem \ref{main result})  and give its proof in Section \ref{section proof main result}. Section \ref{section fundamental equations} is devoted to the equations of Gauss, Ricci and Codazzi and their relations to the Killing type equation in Theorem \ref{main result}, and also to a fundamental theorem for the submanifold theory in $G/H.$ We then study the special case of $\mathbb{H}^3$ in Section \ref{section H3} and the case of $SL_n(\C)/SU(n)$ with $n\geq 3$ in Section \ref{section n geq 3}. Finally, an appendix on the representation of the skew-symmetric operators in the Clifford algebra and on the metric on $Cl(\mathfrak{g})$ ends the paper.

\section{Preliminaries on the homogeneous space $G/H$}\label{section preliminaries}
We consider a finite-dimensional complex semi-simple Lie group $G,$ and a compact real form $H$ of $G.$ Moreover $G$ is simply connected and $H$ is connected.
\subsection{The symmetric Lie algebra $(\mathfrak{g},\mathfrak{h},\sigma)$}
Let us denote by $\mathfrak{g}$ and $\mathfrak{h}$ the Lie algebras of $G$ and $H,$ and consider the decomposition
\begin{equation}\label{decomposition reductif} 
\mathfrak{g}=\mathfrak{h}\oplus \mathfrak{m}
\end{equation}
with $\mathfrak{m}=i\mathfrak{h}.$ The Lie bracket $[.,.]:\ \mathfrak{g}\times\mathfrak{g}\rightarrow\mathfrak{g}$ is $\C$-linear and satisfies
$$[\mathfrak{h},\mathfrak{h}]\subset \mathfrak{h},\hspace{.5cm} [\mathfrak{m},\mathfrak{m}]\subset \mathfrak{h},\hspace{.5cm} [\mathfrak{m},\mathfrak{h}]\subset\mathfrak{m}$$
since the complex Lie algebra $\mathfrak{g}$ is the complexification of the real Lie algebra $\mathfrak{h}.$ Let $Ad:\ G\rightarrow Aut(\mathfrak{g})$
be the adjoint map, and consider its differential at the identity $e$ of $G$
\begin{eqnarray}
ad:=dAd_e:\ \mathfrak{g}&\rightarrow& End(\mathfrak{g})\label{def ad lie algebra}\\
u&\mapsto&(X\in\mathfrak{g}\mapsto [u,X]\in\mathfrak{g}).\nonumber
\end{eqnarray}
Since $\mathfrak{g}$ is semi-simple, the form
\begin{eqnarray*}
B:\hspace{.3cm}\mathfrak{g}\times \mathfrak{g}&\rightarrow&\C\\
(X,Y)&\mapsto&\lambda\ tr(ad(X)\circ ad(Y)),
\end{eqnarray*}
$\lambda>0,$ is non-degenerate (it is a multiple of the Killing form); it is in fact negative definite on $\mathfrak{h}$ and positive definite on $\mathfrak{m}.$ Let us recall that
\begin{equation}\label{B Ad invariant}
B(Ad(g)(X),Ad(g)(Y))=B(X,Y)
\end{equation}
and
\begin{equation}\label{B ad invariant}
B(ad(u)(X),Y)+B(X,ad(u)(Y))=0
\end{equation}
for all $g\in G$ and $X,Y,u\in\mathfrak{g}.$ In particular, $ad(u):\mathfrak{g}\rightarrow\mathfrak{g}$ is $\C$-linear and skew-symmetric with respect to $B,$ and thus naturally identifies to an element of $\Lambda^2\mathfrak{g}.$ Setting 
\begin{eqnarray}\label{def symetrie}
\sigma:\hspace{1cm} \mathfrak{g}=\mathfrak{h}\oplus i\mathfrak{h}&\rightarrow&\mathfrak{g}=\mathfrak{h}\oplus i\mathfrak{h}\\
X+iY&\mapsto& X-iY,\nonumber
\end{eqnarray}
the triplet $(\mathfrak{g},\mathfrak{h},\sigma)$ is a \textit{symmetric Lie algebra}. It is in fact an orthogonal symmetric Lie algebra (since $B$ is negative definite on $\mathfrak{h}$), of non-compact type (since $B$ is positive definite on $\mathfrak{m}$); a classification of the irreducible orthogonal symmetric Lie algebras may be found in \cite{KN2}, Chapter XI, Theorem 8.5, and the case studied here corresponds to the fourth and last case in that classification.

\subsection{The canonical connection and the Maurer-Cartan form on $G/H$}
The canonical projection $\pi:\ G\rightarrow G/H$ is a principal bundle of structural group $H;$ if $\omega_G\in\Omega^1(G,\mathfrak{g})$ stands for the Maurer-Cartan form of $G$ and $p_1$ is the projection onto the first factor $\mathfrak{h}$ in (\ref{decomposition reductif}), we consider \textit{the canonical connection form}
\begin{equation}\label{connection alpha}
\alpha:=p_1\circ\omega_G\ \in\Omega^1(G,\mathfrak{h})
\end{equation}
on the principal bundle $G\rightarrow G/H.$ The bundle 
\begin{equation}\label{def TGH}
\displaystyle{T(G/H)=G\times_{Ad_{|H}}\mathfrak{m}}
\end{equation}
 may be regarded as a sub-bundle of the trivial bundle $G/H\times\mathfrak{g}$ using the map
\begin{eqnarray}
T(G/H)&\rightarrow&G/H\times\mathfrak{g}\label{map TGH subbundle}\\
\ [g,u]&\mapsto&(gH,Ad(g)(u)).\nonumber
\end{eqnarray}
This map is well defined since, for all $g\in G,$ $u\in\mathfrak{m}$ and $h\in H,$
$$Ad(gh)(Ad_{|H}(h^{-1})(u))=Ad(g)(u).$$
The Maurer-Cartan form on $G/H$ introduced in \cite{Burstall} is the 1-form $\beta\in\Omega^1(G/H,\mathfrak{g})$ given by
\begin{equation}\label{def beta}
\beta([g,u]):=Ad(g)(u)\ \in\mathfrak{g}
\end{equation}
for all $[g,u]\in T(G/H).$ For $g\in G,$ we set $\mathfrak{m}_g:=Ad(g)(\mathfrak{m})$ and $\mathfrak{h}_g:=Ad(g)(\mathfrak{h}),$ and consider the splitting $\mathfrak{g}=\mathfrak{m}_g\oplus \mathfrak{h}_g$ together with the corresponding projections $p_{\mathfrak{m}_g}:\mathfrak{g}\rightarrow \mathfrak{m}_g$ and $p_{\mathfrak{h}_g}:\mathfrak{g}\rightarrow \mathfrak{h}_g.$ We have by (\ref{map TGH subbundle})
\begin{equation}\label{trivial TGH}
T(G/H)\simeq \bigsqcup_{gH\ \in G/H}\mathfrak{m}_g\ \subset G/H\times\mathfrak{g}.
\end{equation}
It is proved in \cite{Burstall} that the canonical covariant derivative $\nabla^o$ on $T(G/H)$ is the projection of the usual derivative in $\mathfrak{g}$ onto the sub-bundle $T(G/H):$ if $X$ belongs to $T_{gH}(G/H)$ and $Y:G/H\rightarrow\mathfrak{g}$ is a section of $T(G/H)$ then
\begin{equation}\label{cov deriv T}
\nabla^o_XY=p_{\mathfrak{m}_g}\left(\partial_XY\right).
\end{equation}
Analogously, let us consider the bundle
\begin{equation}\label{def iTGH}
iT(G/H):=G\times_{Ad_{|H}}\mathfrak{h}.
\end{equation}
It may also be regarded as a sub-bundle of $G/H\times\mathfrak{g},$ using the map (\ref{map TGH subbundle}) extended by $\C$-linearity:
\begin{equation}\label{trivial iTGH}
iT(G/H)\simeq \bigsqcup_{gH\ \in G/H}\mathfrak{h}_g\ \subset G/H\times\mathfrak{g}.
\end{equation}
A section $Y$ of $iT(G/H)$ may thus be considered as a map $Y:G/H\rightarrow\mathfrak{g},$ and we may set, for $X$ belonging to $T_{gH}(G/H),$
\begin{equation}\label{cov deriv iT}
\nabla^o_XY=p_{\mathfrak{h}_g}\left(\partial_XY\right).
\end{equation}
Formulas (\ref{cov deriv T}) and (\ref{cov deriv iT}) define a covariant derivative $\nabla^o$ on 
\begin{equation}\label{trivial TCGH}
T^{\C}(G/H):=T(G/H)\oplus iT(G/H)\simeq G/H\times\mathfrak{g},
\end{equation}
which is $\C$-linear: $\nabla^o_X(iY)=i(\nabla^o_XY)$ for $X\in T^{\C}(G/H)$ and $Y\in \Gamma(T^{\C}(G/H)).$

\subsection{The adjoint map on $G/H$}\label{section def ad}
Let us consider the form $ad\circ\beta\in\Omega^1(G/H,\Lambda^2\mathfrak{g})$
where $\beta\in\Omega^1(G/H,\mathfrak{g})$ is the Maurer-Cartan form on $G/H$ defined in (\ref{def beta}) and $ad:\mathfrak{g}\rightarrow\Lambda^2\mathfrak{g}$ is the adjoint map defined in (\ref{def ad lie algebra}). For sake of simplicity, we will call that form the \emph{adjoint map on $G/H$} and we will still denote it by $ad.$  For $Y\in \Gamma(T(G/H)),$ and in view of  (\ref{trivial TCGH}), the map $ad(Y)=ad\circ\beta(Y):G/H\rightarrow\Lambda^2\mathfrak{g}$ may be regarded as a bundle map
$$ad(Y):\ T^{\C}(G/H)\rightarrow\ T^{\C}(G/H)$$
which is skew-symmetric with respect to the form $B.$ Since
\begin{equation}\label{bracket mg properties}
[\mathfrak{m}_g,\mathfrak{m}_g]\subset\mathfrak{h}_g\hspace{.5cm}\mbox{and}\hspace{.5cm}[\mathfrak{m}_g,\mathfrak{h}_g]\subset\mathfrak{m}_g
\end{equation}
for all $g\in G,$ it exchanges the bundles $T(G/H)$ and $iT(G/H)$ (see (\ref{trivial TGH}) and (\ref{trivial iTGH})). The adjoint map
$$ad\ \in \Gamma(T^*(G/H)\otimes\Lambda_{\C}^2(G/H))$$
may then be extended by $\C$-linearity to an object
$$ad\ \in \Gamma(T_{\C}^*(G/H)\otimes\Lambda_{\C}^2(G/H)).$$
Let us note that we may equivalently define $ad$ by the formula
\begin{eqnarray}
ad:\hspace{1cm} T^{\C}(G/H)=G\times_{Ad_{|H}}\mathfrak{g}&\rightarrow& \Lambda^2_{\C}(G/H)=G\times_{Ad_{|H}}\Lambda^2\mathfrak{g}\label{def equiv ad}\\
X=[g,u]&\mapsto& ad(X):=[g,ad(u)]\nonumber
\end{eqnarray}
where, for $u$ belonging to $\mathfrak{g},$ $ad(u)\in\Lambda^2\mathfrak{g}$ represents the endomorphism $v\in\mathfrak{g}\mapsto [u,v]\in\mathfrak{g}$ (which is the skew-symmetric with respect to $B$). This map is well defined since the map $u\in \mathfrak{g}\mapsto ad(u)\in\Lambda^2\mathfrak{g}$ is $Ad$-equivariant: the natural action of $Ad(h)$ on $\eta\in\Lambda^2\mathfrak{g}$ regarded as a map $\eta:\mathfrak{g}\rightarrow\mathfrak{g}$ is by conjugation,
\begin{equation}\label{def Ad conjug}
Ad(h)\cdot\eta=Ad(h)\circ \eta\circ Ad(h)^{-1},
\end{equation}
and we have, for all $h\in H$ and $u,v\in \mathfrak{g},$
\begin{equation}\label{pty bracket Ad inv}
Ad(h)([u,v])=[Ad(h)(u),Ad(h)(v)],
\end{equation}
which implies that
$$Ad(h)([u,Ad(h)^{-1}(v)])=[Ad(h)(u),v],$$
for all $h\in H$ and $u,v\in \mathfrak{g},$ or equivalently
$$Ad(h)\cdot ad(u)=ad(Ad(h)(u))$$
for all $h\in H$ and $u\in\mathfrak{g}.$ The following identity expresses that the operator $ad$ is parallel:
\begin{lem}\label{lemma nablao ad commute}
For all $X\in T^{\C}(G/H)$ and $Y\in\Gamma(T^{\C}(G/H)),$ we have
\begin{equation}\label{nablao ad commute}
\nabla_X^o\circ ad(Y)-ad(Y)\circ\nabla_X^o=ad(\nabla_X^oY)
\end{equation}
where $ad(Y)$ is considered as a bundle map $T^{\C}(G/H)\rightarrow T^{\C}(G/H).$
\end{lem}

\begin{proof}
It is sufficient to prove that, for $X\in T(G/H)$ and $Y,Z\in\Gamma(T(G/H)),$
\begin{equation}\label{comp ad nabla XYZ}
\nabla^o_X(ad(Y)(Z))-ad(Y)(\nabla^o_XZ)=ad(\nabla_X^oY)(Z).
\end{equation}
Recalling (\ref{def TGH}), we fix $x_o\in G/H,$ $X\in T_{x_o}G/H$ and a local section $s$ of $G\rightarrow G/H$ in the neighborhood of $x_o$ such that $\partial_X s(x_o)$ is horizontal with respect to the canonical connection $\nabla^o$. If $\underline{Y}:G/H\rightarrow \mathfrak{m}$ and $\underline{Z}:G/H\rightarrow  \mathfrak{m}$ are such that $Y=[s,\underline{Y}]$ and $Z=[s,\underline{Z}],$ (\ref{comp ad nabla XYZ}) reduces to
$$\partial_X(ad(\underline{Y})(\underline{Z}))-ad(\underline{Y})(\partial_X\underline{Z})=ad(\partial_X\underline{Y})(\underline{Z}),$$
which is the usual Leibniz property.
\end{proof}

\section{A model for $G/H$ into $Spin(\mathfrak{g})$}\label{section cartan embedding}
We introduce in this section a model of $G/H$ into the spin group $Spin(\mathfrak{g})$: it will be defined as the image of the composition of the Cartan embedding $c:G/H\rightarrow G$ with the natural lift $\widetilde{Ad}:G\rightarrow Spin(\mathfrak{g})$ of the adjoint map $Ad:G\rightarrow SO(\mathfrak{g}).$ This model will be important to the spinorial representation theorem: it will be used to give an explicit representation formula of the immersion in terms of the spinor field (formula (\ref{weierstrass}) in Theorem \ref{main result} below). Section \ref{section clifford spin} introduces notation which is important for the rest of the paper, while Sections \ref{section adjoint in spin} and \ref{section model adjoint} are more specific: they introduce in details the model of $G/H$ in $Spin(\mathfrak{g})$ (in two steps), and the reader may wish to skip them in a first reading (the details will really be needed only for the proof of Theorem \ref{main result}).

\subsection{The Clifford algebra and the spin group of $\mathfrak{g}$}\label{section clifford spin}
We consider the complex Clifford algebra constructed from the Lie algebra $\mathfrak{g}$ and the bilinear form $B$
$$Cl(\mathfrak{g}):=\oplus_{k\in\N}\ \mathfrak{g}^{\otimes k}/\mathcal{I}$$
where $\mathcal{I}$ is the ideal generated by elements of the form $u\otimes v+v\otimes u=-2B(u,v),$ and the corresponding spin group
$$Spin(\mathfrak{g}):=\{u_1\cdot u_2\cdots u_{2k}:\ u_i\in\mathfrak{g},\ B(u_i,u_i)=1\}\ \subset Cl(\mathfrak{g}).$$
We note that $Cl(\mathfrak{g})=Cl(\mathfrak{m})\otimes\C=Cl(\mathfrak{h})\otimes\C$ since $(\mathfrak{g},B)$ is both the complexification of $(\mathfrak{m},B_{|\mathfrak{m}})$ and of $(\mathfrak{h},B_{|\mathfrak{h}})$ (see Proposition 1.10 in \cite{BHMMM}). We also define
$$Spin(\mathfrak{m}):=\{u_1\cdot u_2\cdots u_{2k}:\ u_i\in\mathfrak{m},\ B(u_i,u_i)=1\}\ \subset Spin(\mathfrak{g}).$$
Note that there is a natural involution $\sigma$ on $Cl(\mathfrak{g})$ defined by
\begin{eqnarray}
\sigma:\hspace{1cm}Cl(\mathfrak{g})=Cl(\mathfrak{h})\otimes\C&\rightarrow& Cl(\mathfrak{g})=Cl(\mathfrak{h})\otimes\C\label{def sigma clifford}\\
u=\xi\otimes z&\mapsto&\sigma(u):=\xi\otimes \overline{z},\nonumber
\end{eqnarray}
which extends the symmetry $\mathfrak{g}\rightarrow\mathfrak{g}$ introduced in (\ref{def symetrie}).

\subsection{The adjoint group in $Spin(\mathfrak{g})$}\label{section adjoint in spin}
Let us denote by $SO(\mathfrak{g})$ the group of isomorphisms $\mathfrak{g}\rightarrow\mathfrak{g}$ which preserve $B$ and of determinant 1; there is a double covering $Spin(\mathfrak{g})\rightarrow SO(\mathfrak{g}).$ Since $G$ is simply connected, the adjoint map $Ad:G\rightarrow SO(\mathfrak{g})$ has a unique lift
\begin{equation}\label{def Ad tilde}
\widetilde{Ad}:\hspace{.3cm} G\rightarrow Spin(\mathfrak{g})
\end{equation}
which is also a morphism of groups. Its differential at the unit element of $G$ is
$$\widetilde{ad}=\frac{1}{2}ad:\hspace{.3cm}\mathfrak{g}\rightarrow \Lambda^2\mathfrak{g}\ \subset Cl(\mathfrak{g});$$
this is a Lie morphism, if we consider the Lie bracket on $Cl(\mathfrak{g})$ given by the commutator
\begin{equation}\label{bracket clg}
[\eta,\eta']=\eta\cdot\eta'-\eta'\cdot\eta
\end{equation}
for all $\eta,$ $\eta'\in Cl(\mathfrak{g});$ that is, 
\begin{equation}\label{tilde ad morphism}
\widetilde{ad}\ [X,Y]=[\widetilde{ad}(X),\widetilde{ad}(Y)]
\end{equation}
for all $X,Y\in\mathfrak{g},$ where the bracket in the right hand side is defined in (\ref{bracket clg}). It will be convenient to endow the group $Spin(\mathfrak{g})$ with the left invariant metric $B'$ such that
\begin{equation}\label{Bp ad B}
B'\left(\widetilde{ad}(X),\widetilde{ad}(Y)\right)=\frac{1}{4}B(X,Y)
\end{equation}
for all $X,Y\in\mathfrak{g}$; by Lemma \ref{lem2 ap2} in the appendix, this is the metric $B'=-2\lambda B$ where $B:Cl(\mathfrak{g})\times Cl(\mathfrak{g})\rightarrow\C$ is the natural extension of $B:\mathfrak{g}\times\mathfrak{g}\rightarrow\C$ to the Clifford algebra; in particular $B'$ is in fact defined on $Cl(\mathfrak{g})$ and is invariant by left and right multiplication by $Spin(\mathfrak{g})$ (Lemma \ref{lem1 ap2} in the appendix). Since $\widetilde{ad}$ is one-to-one ($G$ is semi-simple), the map (\ref{def Ad tilde}) is an isometric immersion 
$$\widetilde{Ad}:\hspace{.3cm} \left(G,\frac{1}{4}B\right)\rightarrow \left(Spin(\mathfrak{g}),B'\right).$$
Let us consider \emph{the adjoint group}
$$\widetilde{Ad}(G)=\{\widetilde{Ad}(g),\ g\in G\}\hspace{.3cm}\subset Spin(\mathfrak{g}).$$
We will also assume that $\widetilde{Ad}(G)$ is endowed with the metric $B'.$ 

\subsection{A model for $G/H$ into the adjoint group}\label{section model adjoint}
Let us first recall the Cartan embedding of $G/H$ into $G:$ if $\sigma:G\rightarrow G$ is the automorphism such that $d\sigma_e:\mathfrak{g}\rightarrow\mathfrak{g}$ is the involution (\ref{def symetrie}), the map
\begin{eqnarray*}
c:\ G/H&\rightarrow&G\\
gH&\mapsto&g\ \sigma(g^{-1})
\end{eqnarray*}
is an isometric embedding of $G/H$ into $G,$ if $G/H$ is equipped with the metric induced by $B$ and $G$ with the left invariant metric $\frac{1}{4}B$; moreover, $c(G/H)$ is a totally geodesic submanifold of $G;$ see  \cite{H} p. 276. We consider here the composition 
\begin{equation}\label{Adt circ c}
\widetilde{Ad}\circ c:\ G/H\rightarrow Spin(\mathfrak{g}).
\end{equation}
By construction, this is an isometric immersion ($G/H$ is endowed with the metric $B$ and $Spin(\mathfrak{g})$ with the metric $B'$). If $a$ belongs to $Spin(\mathfrak{g})$ we set $a^*:=\sigma(a^{-1})\in Spin(\mathfrak{g}),$ and consider
$$\mathcal{H}:=\{aa^*:\ a\in  \widetilde{Ad}(G)\}\subset \widetilde{Ad}(G).$$
Let us note that $\mathcal{H}$ is the image of $\widetilde{Ad}\circ c$, since, if $a=\widetilde{Ad}(g),$
$$aa^*=\widetilde{Ad}(g)\widetilde{Ad}(g)^*=\widetilde{Ad}(g)\widetilde{Ad}(\sigma(g^{-1}))=\widetilde{Ad}(g\ \sigma (g^{-1}))=\widetilde{Ad}\circ c\ (g)$$ 
($\sigma$ commutes with $\widetilde{Ad}$ since $\sigma(ad(X))=ad(\sigma(X))$ for all $X\in\mathfrak{g}$). Let us assume in general that (\ref{Adt circ c}) is an embedding, so that $\mathcal{H}$ may be regarded as a model of $G/H,$ and briefly show that this is indeed the case for $G=SL_n(\C)$ and $H=SU(n).$ In fact, 
\begin{equation}\label{Ad circ c}
Ad:c(G/H)\rightarrow SO(\mathfrak{g})
\end{equation} 
is an embedding in that case: let us first note that $\sigma(g)=(g^*)^{-1}$ where $g^*$ stands for the conjugate transpose of the matrix $g\in SL_n(\C)$ ($g$ belongs to $SU(n)$ if and only if $\det g=1$ and $(g^*)^{-1}=g$); the center of $SL_n(\C)$ is $\{\lambda I_n,\ \lambda^n=1\}$ and if $Ad(gg^*)=Ad(g'{g'}^*)$ then $gg^*=\lambda g'{g'}^*$ for $\lambda\in\C$ such that $\lambda^n=1;$ since the traces of $gg^*$ and  $g'{g'}^*$ are positive (the trace of $gg^*$ is $\sum_{ik}g_{ik}\overline{g_{ik}}>0$), $\lambda=1,$ which implies that (\ref{Ad circ c}) is one-to-one; moreover, (\ref{Ad circ c}) is proper: if $Ad(x_k)$ is a bounded sequence in $SO(\mathfrak{g})$ with $x_k\in c(G/H),$ then, writing $x_k=q_k\ d_k\ q_k^{-1}$ with $q_k\in SU(n)$ and $d_k=(\lambda_1,\ldots,\lambda_n)$ a diagonal matrix with real entries ($x_k$ is hermitian), the sequence $Ad(d_k)$ is bounded too, which implies that $d_k X d_k^{-1}$ is bounded for all given $X\in M_n(\C);$ this is turn implies that the quotients $\lambda_i/\lambda_j$ are bounded (taking for instance for $X$ the matrix with all the entries equal to 1), and since $\det(d_k)=\lambda_1\cdots\lambda_n=1$ that all the $\lambda_i$'s are bounded; $x_k$ is thus bounded too, and the result follows.

Note that in general the model $\mathcal{H}$ is a totally geodesic submanifold of $\widetilde{Ad}(G),$ since so is $c(G/H)$ in $G$ and $\widetilde{Ad}$ is an isometry. Note also that the isometry on $G/H$ given by the left multiplication by an element $b\in G$ corresponds to the transformation $aa^*\mapsto b(aa^*)b^*$ in $\mathcal{H}.$

Let us finally describe the canonical connection of $G/H$ in the model $\mathcal{H}.$ Let us consider the Lie algebra $\widetilde{ad}(\mathfrak{g})$ of $\widetilde{Ad}(G),$ and the natural trivialization
\begin{eqnarray}
T\widetilde{Ad}(G)&\rightarrow&\widetilde{Ad}(G)\times \widetilde{ad}(\mathfrak{g})\label{triv TAdt}\\
Z\in T_m\widetilde{Ad}(G)&\mapsto& (m,\underline{Z})\nonumber
\end{eqnarray}
where $\underline{Z}\in  \widetilde{ad}(\mathfrak{g})=T_1\widetilde{Ad}(G)$ is such that $Z=d(L_m)_1\underline{Z}$ ($L_m$ is the left multiplication by $m$ in $\widetilde{Ad}(G)$). Since $\widetilde{Ad}(G)$ is endowed with the left invariant metric $B',$ the Levi-Civita connection $\widetilde{\nabla}^o$ of $\widetilde{Ad}(G)$ is given in the trivialization (\ref{triv TAdt}) by
\begin{equation}\label{def nablao tilde}
\underline{\widetilde{\nabla}^o_XY}=\partial_X\underline{Y}+\frac{1}{2}[\underline{X},\underline{Y}]
\end{equation}
for all $X,Y\in \Gamma(T\widetilde{Ad}(G)),$ where the bracket is here the bracket defined in (\ref{bracket clg}); indeed, the Koszul formula for left invariant vector fields $X,Y,Z$ implies that
$$2B'(\widetilde{\nabla}^o_XY,Z)=B'([X,Y],Z)-B'([Y,Z],X)+B'([Z,X],Y),$$
which reduces to $B'([X,Y],Z)$ by the property (\ref{B Cl ad inv}) in Lemma \ref{lem1 ap2}. Since $\mathcal{H}$ is totally geodesic in $\widetilde{Ad}(G),$ $\widetilde{\nabla}^o$ is also the canonical connection of $G/H$ in this model.

\section{The spinor bundle on $G/H$}\label{section spinor bundle GH}
\subsection{Definition of the spinor and Clifford bundles}
Following \cite{Bar1}, a spin structure on the homogeneous space $G/H$ is a representation 
\begin{equation}\label{spin structure}
\widetilde{{Ad}_{|H}}:\hspace{.3cm}H\rightarrow Spin(\mathfrak{m})
\end{equation}
which is a lift of the isotropy representation $Ad_{|H}: H\rightarrow SO(\mathfrak{m}).$ We consider here $\widetilde{{Ad}_{|H}}:=\widetilde{Ad}_{|H}$ where $\widetilde{Ad}:G\rightarrow Spin(\mathfrak{g})$ is the morphism lifting the adjoint representation
\begin{equation}\label{adj repr G}
Ad:\hspace{.3cm}G\rightarrow SO(\mathfrak{g})
\end{equation}
(note that $\widetilde{Ad}_{|H}:H\rightarrow Spin(\mathfrak{g})$ takes values in $Spin(\mathfrak{m})$ since it is a lift of $Ad_{|H}: H\rightarrow SO(\mathfrak{m})$). We may then define spinor and Clifford bundles on $G/H:$ let us consider the representation  
\begin{equation}\label{def rho}
\rho:H\rightarrow Aut(Cl(\mathfrak{g}))
\end{equation}
given by the composition of $\widetilde{Ad}_{|H}$ with the representation of $Spin(\mathfrak{m})$ on $Cl(\mathfrak{g})$ given by multiplication on the left. We define
$$\Sigma:=G\times_{\rho}Cl(\mathfrak{g})$$
and
$$U\Sigma:=G\times_{\rho}\widetilde{Ad}(G)\ \subset\Sigma$$
($\widetilde{Ad}(G)$ is a subgroup of $Spin(\mathfrak{g})\subset Cl(\mathfrak{g})),$ together with 
$$Cl_{\Sigma}:=G\times_{Ad_{|H}} Cl(\mathfrak{g})$$
where $Ad: G\rightarrow Aut(Cl(\mathfrak{g}))$ is the natural extension of the adjoint representation (\ref{adj repr G}) to the Clifford algebra (it is well defined by (\ref{B Ad invariant})). They are bundles over $G/H.$ The bundle $\Sigma$ is very similar to the usual spinor bundle, with the difference that the representation (\ref{def rho}) is not irreducible in general. Note that it is defined using the entire Lie algebra $\mathfrak{g},$ instead of $\mathfrak{m}$ only: this will be the key to obtain a special global section in Section \ref{section global section}. The bundle $U\Sigma$ is a sub-bundle of $\Sigma$ and will be interpreted as \emph{the bundle of unit spinors} on $G/H:$ it will be important for the formulation of the spinorial representation theorem, since a normalization of the spinor fields is required to represent \emph{isometric} immersions. As in the usual construction in spin geometry, there is a Clifford action
\begin{eqnarray*}
Cl_{\Sigma}\times\Sigma&\rightarrow&\Sigma\\
(\eta,\varphi)&\mapsto&\eta\cdot\varphi.
\end{eqnarray*}
It is such that, if $\eta$ and $\varphi$ are respectively represented by $[\eta]$ and $[\varphi]\in Cl(\mathfrak{g})$ in some $g\in G,$ then $\eta\cdot\varphi$ is represented by $[\eta]\cdot[\varphi]$ in $g.$ This action is well defined since, for $[\eta],$ $[\varphi]\in Cl(\mathfrak{g})$ and $h\in H,$
\begin{eqnarray*}
Ad(\rho(h))([\eta])\cdot \rho(h)([\varphi])&=&\widetilde{Ad}(h)\cdot[\eta]\cdot\widetilde{Ad}(h)^{-1}\cdot\widetilde{Ad}(h)\cdot[\varphi]\\
&=&\widetilde{Ad}(h)\cdot[\eta]\cdot[\varphi]\\
&=&\rho(h)([\eta]\cdot[\varphi]).
\end{eqnarray*}

Let us note that the tangent bundle $T^{\C}(G/H)$ is naturally a sub-bundle of $Cl_{\Sigma},$ by using the map
\begin{eqnarray*}
T^{\C}(G/H)=G\times_{Ad_{|H}}\mathfrak{g}&\rightarrow&Cl_{\Sigma}=G\times_{Ad_{|H}} Cl(\mathfrak{g})\\
\ [g,u]&\mapsto&[g,j(u)]
\end{eqnarray*} 
where $j:\mathfrak{g}\rightarrow Cl(\mathfrak{g})$ is the natural inclusion. Similarly, $\Lambda^2_{\C}(G/H)$ is also a sub-bundle of $Cl_{\Sigma}$ by the natural map
\begin{equation}\label{Lambda2 Cl}
\Lambda_{\C}^2(G/H)=G\times_{Ad_{|H}}\Lambda^2\mathfrak{g}\ \ \subset\ Cl_{\Sigma}=G\times_{Ad_{|H}} Cl(\mathfrak{g}).
\end{equation}

Finally, the spinor bundle $\Sigma$ is endowed with the connection $\nabla^o$ associated to the canonical connection form $\alpha\in\Omega^1(G,\mathfrak{h})$ on $G\rightarrow G/H$ introduced in (\ref{connection alpha}).
\begin{rem} 
The adjoint map on $G/H$ may be regarded as a section of $T_{\C}^*(G/H)\otimes Cl_{\Sigma}$: by (\ref{Lambda2 Cl}), the adjoint map $ad\in \Gamma(T_{\C}^*(G/H)\otimes \Lambda^2_{\C}(G/H))$ defined in Section \ref{section def ad} naturally belongs to $\Gamma(T_{\C}^*(G/H)\otimes Cl_{\Sigma}).$ In particular, for all $Y\in T^{\C}(G/H),$ $ad(Y)$ belongs to $Cl_{\Sigma}$ and thus naturally acts on $\Sigma.$
\end{rem}
\subsection{The bilinear map $\langle\langle.,.\rangle\rangle:\Sigma\times\Sigma\rightarrow Cl(\mathfrak{g})$}
There is a natural bilinear map
\begin{eqnarray*}
\langle\langle.,.\rangle\rangle:\hspace{1cm} \Sigma\times\Sigma&\rightarrow& Cl(\mathfrak{g})\\
(\varphi,\psi)&\mapsto&\langle\langle\varphi,\psi\rangle\rangle:=\tau[\psi][\varphi]
\end{eqnarray*}
where $[\varphi]$ and $[\psi]\in Cl(\mathfrak{g})$ represent $\varphi$ and $\psi$ in some frame $g\in G,$ and $\tau:Cl(\mathfrak{g})\rightarrow Cl(\mathfrak{g})$ is the involution which reverses the order of the terms  
$$\tau(v_1\cdot v_2\cdots v_k)=v_k\cdots v_2\cdot v_1$$
for all $v_1,v_2,\ldots,v_k\in\mathfrak{g}.$ The map $\langle\langle.,.\rangle\rangle$ is well defined since
$$Spin(\mathfrak{m})\ \subset\{a\in Cl(\mathfrak{g}):\ \tau(a)a=1\}.$$ 
It satisfies the following properties:

\begin{lem}\label{lemma properties tau 1}
For all $\varphi,\psi\in\Gamma(\Sigma)$ and $X\in \Gamma(T(G/H)),$
\begin{equation}\label{scalar product property1}
\langle\langle \varphi,\psi\rangle\rangle=\tau\langle\langle\psi,\varphi\rangle\rangle
\end{equation}
and
\begin{equation}\label{scalar product property2}
\langle\langle X\cdot\varphi,\psi\rangle\rangle=\langle\langle\varphi,X\cdot \psi\rangle\rangle.
\end{equation}
\end{lem}
\begin{proof}
We have
$$\langle\langle \varphi,\psi\rangle\rangle=\tau[\psi]\ [\varphi]=\tau(\tau[\varphi]\ [\psi])=\tau\langle\langle\psi,\varphi\rangle\rangle$$
and
$$\langle\langle X\cdot\varphi,\psi\rangle\rangle=\tau[\psi]\ [X][\varphi]=\tau([X][\psi])[\varphi]=\langle\langle\varphi,X\cdot \psi\rangle\rangle$$
where $[\varphi],$ $[\psi]$ and $[X]\ \in Cl(\mathfrak{g})$ represent $\varphi,$ $\psi$ and $X$ in some given frame $g\in G.$
\end{proof}
\begin{lem}\label{lemma properties tau 2}
The connection $\nabla^o$ is compatible with the product $\langle\langle.,.\rangle\rangle:$
$$\partial_X\langle\langle\varphi,\varphi'\rangle\rangle=\langle\langle\nabla^o_X\varphi,\varphi'\rangle\rangle+\langle\langle\varphi,\nabla^o_X\varphi'\rangle\rangle$$
for all $\varphi,\varphi'\in\Gamma(\Sigma)$ and $X\in\Gamma(T(G/H)).$ 
\end{lem}
\begin{proof}
If $\varphi=[s,[\varphi]]$ is a section of $\Sigma=G\times_{\rho} Cl(\mathfrak{g}),$ we have
\begin{equation}\label{nabla phi rho}
\nabla^o_X\varphi=\left[s,\partial_X[\varphi]+\rho_*(s^*\alpha(X))([\varphi])\right],\hspace{1cm}\forall X\in\ T(G/H),
\end{equation}
where $\alpha\in\Omega^1(G,\mathfrak{h})$ is the canonical connection form on $G\rightarrow G/H;$ the term $\rho_*(s^*\alpha(X))$ is an endomorphism of $Cl(\mathfrak{g})$ given by the multiplication on the left by an element belonging to $\Lambda^2\mathfrak{g}\subset Cl(\mathfrak{g}),$ still denoted by  $\rho_*(s^*\alpha(X)).$ Such an element satisfies
$$\tau\left( \rho_*(s^*\alpha(X))\right)=-\rho_*(s^*\alpha(X)),$$
and we have
\begin{eqnarray*}
\langle\langle\nabla^o_X\varphi,\varphi'\rangle\rangle+\langle\langle\varphi,\nabla^o_X\varphi'\rangle\rangle&=&\tau\{[\varphi']\}\left(\partial_X[\varphi]+\rho_*(s^*\alpha(X))[\varphi]\right)\\
&&+\tau\left\{\partial_X[\varphi']+\rho_*(s^*\alpha(X))[\varphi']\right\}[\varphi]\\
&=&\tau\{[\varphi']\}\partial_X[\varphi]+\tau\left\{\partial_X[\varphi']\right\}[\varphi]\\
&=&\partial_X\langle\langle\varphi,\varphi'\rangle\rangle.
\end{eqnarray*}
\end{proof}
We finally note that there is a natural action of $\widetilde{Ad}(G)$ on $U\Sigma,$ by right multiplication: for $\varphi=[g,[\varphi]]\in U\Sigma=G\times_{\rho}\widetilde{Ad}(G)$ and $a\in \widetilde{Ad}(G)$ we set
\begin{equation}\label{def right action}
\varphi\cdot a:=[g,[\varphi]\cdot a]\ \in U\Sigma.
\end{equation}
More generally, $Cl(\mathfrak{g})$ naturally acts on $\Sigma$ on the right.
\subsection{The involution $\sigma:\Sigma\rightarrow\Sigma$}
There is a natural involution
\begin{eqnarray*}
\sigma:\hspace{1cm}\Sigma=G\times_{\rho} Cl(\mathfrak{g})&\rightarrow&\Sigma=G\times_{\rho} Cl(\mathfrak{g})\\
\varphi=[g,u]&\mapsto&\sigma(\varphi):= [g,\sigma(u)]
\end{eqnarray*}
where the involution $\sigma$ on $Cl(\mathfrak{g})$ is defined in (\ref{def sigma clifford}) (for sake of simplicity, we use the same letter $\sigma$ to denote the involutions on $\Sigma$ and on $Cl(\mathfrak{g})$). It is well defined since $\sigma(\rho(h)u)=\rho(h)(\sigma(u))$ for all $h\in H$ and $u\in Cl(\mathfrak{g})$ ($\rho(h)$ belongs to $Spin(\mathfrak{m})$ and is thus invariant by $\sigma$ as a product of an even number of vectors of $\mathfrak{m}$).

\subsection{Spinorial geometry of a submanifold in $G/H$}\label{section spin sub}
If $M$ is a spin submanifold of $G/H,$ with normal bundle $E,$ then the Levi-Civita connection on $TM$ and the normal connection on $E$ induce a connection $\nabla$ on $\Sigma_{|M},$ still denoted by $\Sigma$ for sake of simplicity: if $Q_p$ and $Q_q$ respectively denote the bundles of positively oriented orthonormal frames of $TM$ and $E,$ then there exist spin structures $\widetilde{Q}_p\rightarrow Q_p$ and $\widetilde{Q}_q\rightarrow Q_q$ such that, if $\widetilde{Q}=G\times_{\widetilde{Ad}_{|H}} Spin(\mathfrak{m})$ is the spin structure of $G/H,$ there is a map $\widetilde{Q}_p\times_{_M}\widetilde{Q}_q\rightarrow\widetilde{Q}_{|M}$ above the natural map (\emph{the concatenation of bases}) $Q_p\times_{_M} Q_q\rightarrow Q_{|M}$ where $Q=G\times_{Ad_{|H}}SO(\mathfrak{m})$ is the bundle of frames of $G/H$.
By \cite{Bar2,BHMMM}, $\nabla$ and $\nabla^o$ are related by the spinorial Gauss formula
\begin{equation}\label{gauss nabla nablao}
\nabla^o_X\varphi=\nabla_X\varphi+\frac{1}{2}II(X)\cdot\varphi
\end{equation}
for all $X\in TM$ and $\varphi\in\Gamma(\Sigma),$ where, if $e_1,\ldots,e_p$ is an orthonormal basis of $TM,$ 
\begin{equation}\label{def II Clifford}
\frac{1}{2}II(X):=\frac{1}{2}\sum_{i=1}^pe_i\cdot II(X,e_i)\ \in Cl_{\Sigma}
\end{equation}
represents the linear map 
$$\widetilde{II}(X,.):\ Y=Y_M+Y_E\ \in TM\oplus E\mapsto -II^*(X,Y_E)+II(X,Y_M)\in TM\oplus E$$
($II^*(X,.):E\rightarrow TM$ denotes the adjoint of $II(X,.):TM\rightarrow E$); see Lemma \ref{lem3 ap1} in the appendix.

\subsection{A special spinor field on $G/H$}\label{section global section}
Let us begin with a general remark concerning homogeneous bundles on a homogeneous manifold $G/H:$ if $\rho:G\rightarrow GL(V)$ is a linear representation of \emph{the entire group} $G$ and $E_V:=G\times_HV$ is the vector bundle on $G/H$ naturally associated to the representation $\rho_{|H}:H\rightarrow GL(V)$ then the bundle $E_V$ is trivial: the map
\begin{eqnarray*}
E_V=G\times_HV&\rightarrow& G/H\times V\\
\ [g,v]&\mapsto &(gH,\rho(g)(v))
\end{eqnarray*}
is well defined since, for $(g,v)\in G\times V$ and $h\in H,$
$$\left(gh,\rho(h^{-1})(v)\right)\sim_H(g,v)$$
and
$$\rho(gh)\left(\rho(h^{-1})(v)\right)=\rho(g)(v),$$
and is a global trivialization of the bundle. Using in an essential way that the representation (\ref{spin structure}) is the restriction of the representation
$$\widetilde{Ad}:\hspace{.5cm}G\rightarrow Spin(\mathfrak{g})$$
of the entire group $G,$ we set
\begin{eqnarray}
\varphi:\hspace{.5cm}G/H&\rightarrow& U\Sigma \label{def phi}\\
gH&\mapsto&[g,\widetilde{Ad}(g^{-1})].\nonumber
\end{eqnarray}
The section $\varphi$ appears to be the constant section $1_{Cl(\mathfrak{g})}$ in the natural trivialization $\Sigma\simeq G/H\times Cl(\mathfrak{g})$ described above ($\rho$ is here the composition of $\widetilde{Ad}:G\rightarrow Spin(\mathfrak{g})$ with the multiplication on the left on $Cl(\mathfrak{g})$).
Let us verify directly that it is well defined: since
$$\widetilde{Ad}((gh)^{-1})=\widetilde{Ad}(h)^{-1}\widetilde{Ad}(g^{-1})=\rho(h)^{-1}(\widetilde{Ad}(g^{-1}))$$
we have
$$[gh,\widetilde{Ad}((gh)^{-1})]= [g,\widetilde{Ad}(g^{-1})]\ \mbox{in}\ U\Sigma$$
for all $g\in G$ and $h\in H,$ and (\ref{def phi}) defines a global spinor field $\varphi\in\Gamma(U\Sigma).$
\begin{prop}
The spinor field $\varphi$ satisfies the Killing type equation
\begin{equation}\label{killing equation nablao}
\nabla^o_X\varphi=-\frac{1}{2}ad(X)\cdot\varphi.
\end{equation}
for all $X\in T(G/H).$
\end{prop}
\begin{proof}
Let us fix a local section $s$ of $G\rightarrow G/H,$ and consider $[\varphi]\in Spin(\mathfrak{g})$ such that $\varphi=[s,[\varphi]].$ By definition of the covariant derivative $\nabla^o,$ we have, for all $X\in T(G/H),$
$$\nabla^o_X\varphi=\left[s,\partial_X[\varphi]+d\widetilde{Ad}(s^*\alpha(X))[\varphi]\right],$$
where $\alpha$ is the canonical connection form on $G$ defined in (\ref{connection alpha}). Since $\alpha$ is the projection to $\mathfrak{h}$ of the Maurer-Cartan form of $G,$ and since $[\varphi]=\widetilde{Ad}(s^{-1}),$ we get
\begin{equation}\label{first nablao}
\nabla^o_X\varphi=\left[s,\partial_X\widetilde{Ad}(s^{-1})+d\widetilde{Ad}\left\{(s^{-1}\partial_Xs)_{\mathfrak{h}}\right\}\widetilde{Ad}(s^{-1})\right]
\end{equation}
where the sub-index $\mathfrak{h}$ means that we take in the decomposition (\ref{decomposition reductif}) the component of the vector belonging to $\mathfrak{h}.$ Since $\widetilde{Ad}(s)\widetilde{Ad}(s^{-1})=1_{Cl(\mathfrak{g})},$ we get 
\begin{equation}\label{formula s-1}
\partial_X\widetilde{Ad}(s^{-1})=-\widetilde{Ad}(s^{-1})\ \partial_X\widetilde{Ad}(s)\ \widetilde{Ad}(s^{-1})
\end{equation} 
where the product is the product in $Cl(\mathfrak{g}).$ Moreover,
\begin{equation}\label{formula drho}
\widetilde{Ad}(s^{-1})\ \partial_X\widetilde{Ad}(s)=d\widetilde{Ad}(s^{-1}\partial_Xs)
\end{equation}
(if $t\mapsto\gamma(t)$ is a path in $G$ which is tangent to $\partial_Xs$ at $t=0$, then
$$\widetilde{Ad}(s^{-1})\widetilde{Ad}(\gamma(t))=\widetilde{Ad}(s^{-1}\gamma(t))$$
which implies (\ref{formula drho}) by derivation). Thus (\ref{formula s-1}) reads
\begin{equation*}
\partial_X\widetilde{Ad}(s^{-1})=-d\widetilde{Ad}(s^{-1}\partial_Xs)\ \widetilde{Ad}(s^{-1}).
\end{equation*}
Plugging this formula in (\ref{first nablao}) we get
\begin{equation}\label{second nablao}
\nabla^o_X\varphi=\left[s,-d\widetilde{Ad}\left\{(s^{-1}\partial_Xs)_{\mathfrak{m}}\right\} \widetilde{Ad}(s^{-1})\right].
\end{equation}
Now $X_{\mathfrak{m}}:=(s^{-1}\partial_Xs)_{\mathfrak{m}}$ is such that $X=[s,X_{\mathfrak{m}}]$ in $T(G/H)=G\times_{Ad_{|H}}\mathfrak{m};$ since
$$-d\widetilde{Ad}\left\{(s^{-1}\partial_Xs)_{\mathfrak{m}}\right\} =-d\widetilde{Ad}(X_{\mathfrak{m}})=-\frac{1}{2}ad(X_{\mathfrak{m}}),$$
the result follows.
\end{proof}
 In view of (\ref{gauss nabla nablao}) and (\ref{killing equation nablao}) the special spinor field $\varphi$ introduced above is a solution of the Killing type equation
 \begin{equation*}
\nabla_X\varphi=-\frac{1}{2}II(X)\cdot\varphi-\frac{1}{2}ad(X)\cdot\varphi
\end{equation*}
for all $X\in TM.$
\begin{rem}\label{rem special spinor}
For the spinor field $\varphi$ defined in (\ref{def phi}), the expression $\langle\langle \sigma(\varphi),\varphi\rangle\rangle$ is the composition of the Cartan embedding 
$$G/H\rightarrow G,\ \ gH\mapsto g\ \sigma(g^{-1})$$
with the representation $\widetilde{Ad}:G\rightarrow Spin(\mathfrak{g}).$ Indeed, at $gH,$ we have $[\varphi]=\widetilde{Ad}(g^{-1})$ and
$$\langle\langle \sigma(\varphi),\varphi\rangle\rangle =\tau[\varphi] \sigma[\varphi]=\tau\widetilde{Ad}(g^{-1})\ \sigma\widetilde{Ad}(g^{-1})=\widetilde{Ad}\left(g\ \sigma(g^{-1})\right),$$
since $\tau\widetilde{Ad}(g^{-1})=\widetilde{Ad}(g^{-1})^{-1}=\widetilde{Ad}(g)$ and $\sigma$ commutes with $\widetilde{Ad}.$
\end{rem}
 
\section{The spinor bundle in the abstract setting}\label{section abstract setting}
We do here the converse constructions, and define the spinor and Clifford bundles (and the various objects defined on them) in the abstract context, i.e. without assuming that $M$ is a submanifold of $G/H$ (if $M$ is a submanifold of $G/H,$ we naturally suppose that these objects are those constructed in Sections \ref{section cartan embedding} and \ref{section spinor bundle GH}). We suppose that $p$ and $q$ are positive integers such that $p+q=\dim(G/H),$ assume that $M$ is a $p$-dimensional Riemannian manifold and $E$ is a real vector bundle of rank $q,$ with a scalar product in the fibers and a compatible connection. We moreover suppose that $M$ and $E$ are spin, with spin structures $\widetilde{Q}_p\rightarrow Q_p$ and $\widetilde{Q}_q\rightarrow Q_q,$ and consider 
$$\widetilde{Q}_{p,q}:=\widetilde{Q}_p\times_{_M}\widetilde{Q}_q\rightarrow Q_p\times_{_M} Q_q.$$
The morphism
$$r_{p,q}:\ Spin(p)\times Spin(q)\rightarrow Spin(\mathfrak{m})$$
associated to a splitting $\mathfrak{m}=\R^p\oplus\R^q$ gives rise to a bundle
$$\widetilde{Q}:=\widetilde{Q}_{p,q}\times_{r_{p,q}} Spin(\mathfrak{m}).$$
We need to suppose that there exists a $H$-principal bundle $\widetilde{Q}_H,$ reduction of the bundle $\widetilde{Q}$ by the map $\widetilde{Ad}_{|H}: H\rightarrow  Spin(\mathfrak{m}).$ Note that this condition is necessary to obtain an immersion of $M$ into $G/H$ (recall the definition of $\widetilde{Q}$ from the $H$-principal bundle $G\rightarrow G/H$ in Section \ref{section spin sub}), and we will see in Remark \ref{rmk rep H3} ii) below that it is trivially satisfied in dimension 3. In accordance with Section \ref{section spinor bundle GH}, we set
$$\Sigma:=\widetilde{Q}_H\times_\rho Cl(\mathfrak{g}),\hspace{.5cm}U\Sigma:=\widetilde{Q}_H\times_\rho \widetilde{Ad}(G)\hspace{.3cm}\mbox{and}\hspace{.3cm}Cl_\Sigma=\widetilde{Q}_H\times_{Ad_{|H}}Cl(\mathfrak{g}).$$
The bundles $\Sigma$ and $Cl_{\Sigma}$ are equipped with connections $\nabla$ induced by the Levi-Civita connection on $TM$ and the given connection on $E$ (these bundles may be regarded as associated to the bundle $\widetilde{Q}$ since $\widetilde{Q}_H$ is a reduction of $\widetilde{Q}$ and the morphisms $\rho:H\rightarrow Aut(Cl(\mathfrak{g}))$ and $Ad_{|H}:H\rightarrow Aut(Cl(\mathfrak{g}))$ factorize through $\widetilde{Ad}_{|H}:H\rightarrow Spin(\mathfrak{m})$ and $Ad_{|H}:H\rightarrow SO(\mathfrak{m})$ respectively). Let us note that 
$$\displaystyle{Cl_{\Sigma}=Cl((TM\oplus E)^{\C})}$$
since 
$$TM\oplus E=\widetilde{Q}_{H}\times_{Ad_{|H}}\mathfrak{m}$$
and $Cl(\mathfrak{m}\otimes\C)=Cl(\mathfrak{g}).$ We moreover naturally construct on these bundles 
$$\langle\langle.,.\rangle\rangle:\Sigma\times\Sigma \rightarrow Cl(\mathfrak{g}),\ \sigma:\Sigma\rightarrow\Sigma\hspace{.3cm}\mbox{and}\hspace{.3cm}ad\in(TM\oplus E)^*\otimes Cl_{\Sigma}$$
corresponding to the objects introduced in Section \ref{section spinor bundle GH}. Let us give some details (they may be skipped in a first reading). We construct:
\begin{enumerate}
\item the map
$$\langle\langle.,.\rangle\rangle:\hspace{.5cm} \Sigma\times\Sigma \rightarrow Cl(\mathfrak{g}),\hspace{.5cm}(\varphi,\psi)\mapsto\tau[\psi][\varphi];$$
it satisfies the properties (\ref{scalar product property1}) and (\ref{scalar product property2}), and is also compatible with the connection $\nabla$ (see the proof of Lemma \ref{lemma properties tau 2});
\item involutions $\sigma$ of $\Sigma$ and $Cl_{\Sigma}$ (naturally constructed from $\sigma:Cl(\mathfrak{g})\rightarrow Cl(\mathfrak{g})$), such that $\sigma(\eta\cdot\varphi)=\sigma(\eta)\cdot\sigma(\varphi)$ for all $\eta\in\Gamma(Cl_{\Sigma})$ and $\varphi\in\Gamma(\Sigma);$ 
\item  the adjoint map $ad\in (TM\oplus E)^*\otimes Cl_{\Sigma}$ by the formula
\begin{eqnarray*}
ad:\hspace{.5cm}TM\oplus E=\widetilde{Q}_H\times_{Ad_{|H}}\mathfrak{m}&\rightarrow& \widetilde{Q}_H\times_{Ad_{|H}}\Lambda^2\mathfrak{g}\subset Cl_{\Sigma}\\
X=[s,u]&\mapsto&ad(X)= [s,ad(u)];
\end{eqnarray*}
it is well defined by (\ref{def Ad conjug}) and (\ref{pty bracket Ad inv}). Note that $ad(X)$ represents an endomorphism of $(TM\oplus E)^{\C};$ more precisely, since $ad(u)(\mathfrak{m})\subset\mathfrak{h}$ and $ad(u)(\mathfrak{h})\subset\mathfrak{m}$ for all $u\in\mathfrak{m},$  $ad(X)$ exchanges $TM\oplus E$ and $i(TM\oplus E),$ that is $ad(X)$ belongs to $i\Lambda^2(TM\oplus E).$ 
\end{enumerate}

We finally assume that a symmetric and bilinear map $II:TM\times TM\rightarrow E$ is given, which satisfies the following necessary compatibility condition (see Lemma \ref{lemma nablao ad commute}): for $X\in TM$ and $Y=Y_M+Y_E\in \Gamma(TM\oplus E),$ setting
$$\widetilde{II}(X,Y):=II(X,Y_M)-II^*(X,Y_E)$$
and
$$\nabla^o_XY:=\nabla_XY+\widetilde{II}(X,Y),$$ 
then
\begin{equation}\label{compatibility nablao ad}
\nabla_X^o\circ ad(Y)-ad(Y)\circ\nabla_X^o=ad(\nabla_X^oY).
\end{equation}
Note that $II$ may be regarded as a section of $T^*M\otimes Cl_{\Sigma},$ as in (\ref{def II Clifford}).
\begin{rem}
1. Equation (\ref{compatibility nablao ad}) is equivalent to the following two equations:
\begin{equation}\label{compatibility II ad}
\widetilde{II}(X,.)\circ ad(Y)-ad(Y)\circ \widetilde{II}(X,.)=ad(\widetilde{II}(X,Y))
\end{equation}
for all $X\in TM$ and $Y\in TM\oplus E$ and 
\begin{equation}\label{compatibility nabla ad}
\nabla_X\circ ad(Y)-ad(Y)\circ\nabla_X=ad(\nabla_XY)
\end{equation}
for all $X\in TM$ and $Y\in\Gamma(TM\oplus E).$ Indeed (\ref{compatibility II ad}) (and thus also (\ref{compatibility nabla ad})) follows from (\ref{compatibility nablao ad}) if we choose $Y\in\Gamma(TM\oplus E)$ such that $\nabla Y=0$ at some given point.

2. The following observation will be useful in Section \ref{section n geq 3}: the map $B:Cl(\mathfrak{g})\times Cl(\mathfrak{g})\rightarrow\C$ introduced in Appendix \ref{appendix B} (the natural extension to the Clifford algebra of $B:\mathfrak{g}\times\mathfrak{g}\rightarrow\C$) is $Ad$-invariant, and thus induces a map $B:\ Cl_{\Sigma}\times Cl_{\Sigma}\rightarrow\C;$ it is bilinear and symmetric and coincides with the metric on $TM\oplus E\subset Cl_{\Sigma}.$ 
\end{rem}

\section{Statement of the main result}\label{section main result}
We keep here the definitions, notations and hypotheses of the previous section and state the main result of the paper. For sake of clarity we briefly summarize all the required assumptions (and refer to the previous section for the constructions and details): the symmetric space $G/H$ is fixed, $M$ is an abstract Riemannian manifold of dimension $p$ and $E$ is a real vector bundle on $M$ of rank $q$ such that $p+q=dim(G/H),$ with a fibre metric and a connection compatible with the metric. We suppose that $M$ and $E$ are spin, with given spin structures, and further do the following three assumptions:
\begin{enumerate}
\item[(H1)]\label{H1}
a reduction $\widetilde{Q}_H$ of the spin structure $\widetilde{Q}$ on $TM\oplus E$ is given;
\\
\item[(H2)]\label{H2}
a map $II:TM\times TM\rightarrow E$ is given, which is symmetric, bilinear and compatible with the map $ad\in\Gamma\left((TM\oplus E)^*\otimes Cl_{\Sigma}\right)$ constructed in the previous section; 
\\
\item[(H3)]\label{H3}
the natural map 
$G/H\rightarrow Spin(\mathfrak{g})$ introduced in Section \ref{section cartan embedding} is an embedding, i.e. its range $\mathcal{H}$ is a model of $G/H$ into $Spin(\mathfrak{g})$.
\\
\end{enumerate}
Under these hypotheses, the following holds:
\begin{thm}\label{main result}
The following two statements are equivalent:
\begin{enumerate}
\item there is a solution $\varphi\in\Gamma(U\Sigma)$ of
\begin{equation}\label{killing equation}
\nabla_X\varphi=-\frac{1}{2}II(X)\cdot\varphi-\frac{1}{2}ad(X)\cdot\varphi
\end{equation}
for all $X\in TM;$
\item there is an isometric immersion $F:\ M\rightarrow G/H$ with normal bundle $E$ and second fundamental form $II.$
\end{enumerate}
Moreover, the isometric immersion is explicitly written in terms of the spinor field by the formula
\begin{equation}\label{weierstrass}
F=\langle\langle\sigma(\varphi),\varphi\rangle\rangle\ \in\mathcal{H}\simeq G/H.
\end{equation}
\end{thm}
Formula (\ref{weierstrass}) (together with (\ref{killing equation})) is interpreted as a generalized Weierstrass representation formula.
\begin{rem}
We briefly comment the hypotheses (H1), (H2) and (H3) of the theorem: let us first recall that (H1) and (H2) are necessary to obtain an immersion of $M$ in $G/H$ and also that (H3) is satisfied in the most important case $G/H=SL_n(\C)/SU(n)$ (see Section \ref{section cartan embedding}); moreover, we will see in Remark \ref{rmk rep H3} ii) below that in dimension 3 (H1) is not an additional requirement and (H2) is satisfied for the model $G/H=SL_2(\C)/SU(2)$ of $\mathbb{H}^3$. However, in general, (H1) and (H2) are certainly strong (but necessary) compatibility assumptions on the data.
 \end{rem}
\section{Proof of theorem \ref{main result}}\label{section proof main result}
\subsection{Proof of $(2) \Rightarrow (1)$}
From the considerations above, an immersion of $M$ into $G/H$ gives rise to a normalized spinor field solution of (\ref{killing equation}): $\varphi$ is the restriction to $M$ of the special spinor field (\ref{def phi}). It remains to prove that formula (\ref{weierstrass}) holds: by Remark \ref{rem special spinor} ii), (\ref{weierstrass}) is the composition of the Cartan embedding $c:G/H\rightarrow G$ with the representation $\widetilde{Ad}:G\rightarrow Spin(\mathfrak{g}),$ in restriction to $M.$ So (\ref{weierstrass}) obviously holds if $G/H$ is identified with its model $\mathcal{H}=(\widetilde{Ad}\circ c)(G/H)$ in $Spin(\mathfrak{g}).$
\subsection{Proof of $(1) \Rightarrow (2)$}
We verify that formula (\ref{weierstrass}) gives the required immersion. We first note the following formula:
\begin{lem}\label{lem F-1dF}
If $\varphi$ is a solution of (\ref{killing equation}) and $F:=\langle\langle\sigma(\varphi),\varphi\rangle\rangle,$ we have, $\forall X\in TM,$
\begin{equation}\label{formula lem dF}
\partial_XF=\langle\langle ad(X)\cdot\sigma(\varphi),\varphi\rangle\rangle.
\end{equation} 
\end{lem}
\noindent \textit{Proof of Lemma \ref{lem F-1dF}}. Using (\ref{killing equation}) we have
\begin{eqnarray*}
\partial_XF&=&\langle\langle \sigma(\nabla_X\varphi),\varphi\rangle\rangle+\langle\langle \sigma(\varphi),\nabla_X\varphi\rangle\rangle\\
&=&-\frac{1}{2}\langle\langle \sigma((II(X)+ad(X))\cdot\varphi),\varphi\rangle\rangle-\frac{1}{2}\langle\langle \sigma(\varphi),(II(X)+ad(X))\cdot\varphi\rangle\rangle.
\end{eqnarray*}
But
$$\langle\langle \sigma((II(X)+ad(X))\cdot\varphi),\varphi\rangle\rangle=\langle\langle (II(X)-ad(X))\cdot\sigma(\varphi),\varphi\rangle\rangle$$
since $\sigma(II(X))=II(X)$ and $\sigma(ad(X))=-ad(X)$ ($II(X)$ belongs to $\Lambda^2 (TM\oplus E)$ and $ad(X)$ to $i\Lambda^2 (TM\oplus E)$), and
$$\langle\langle \sigma(\varphi),(II(X)+ad(X))\cdot\varphi\rangle\rangle=-\langle\langle (II(X)+ad(X))\cdot \sigma(\varphi),\varphi\rangle\rangle$$
by (\ref{scalar product property2}) and since $II(X)+ad(X)$ is a bivector. This gives (\ref{formula lem dF}) i.e. the lemma. 
\\

Using $F=\tau[\varphi] \sigma[\varphi]$ and $F^{-1}=\tau\sigma[\varphi] [\varphi]$ we easily deduce from the lemma the useful formula 
\begin{equation}\label{F^-1dF}
F^{-1}\partial_XF=-\sigma\langle\langle ad(X)\cdot\varphi,\varphi\rangle\rangle.
\end{equation}
The following lemma shows that $F$ is an isometric immersion with normal bundle $E$ and second fundamental form $II,$ and its proof will thus finish the proof of the theorem.
\begin{lem}\label{lem F isom immersion}
i) The map $F:M\rightarrow\mathcal{H}$ is an isometric immersion.
\\ii) Let us denote by $E^F,$ $II^F$ and ${\nabla'}^F$ the normal bundle, the second fundamental form and the normal connection of the immersion $F:M\rightarrow\mathcal{H}.$ The map
\begin{eqnarray}\label{def Phi}
\Phi:\hspace{1cm}E&\rightarrow&E^F\\
Z&\mapsto&\langle\langle ad(Z)\cdot\sigma(\varphi),\varphi\rangle\rangle\nonumber
\end{eqnarray}
is such that
\begin{equation}\label{II preserve}
II^F(F_*X,F_*Y)=\Phi(II(X,Y))
\end{equation}
and
\begin{equation}\label{nabla' preserve}
{\nabla'}^F_{X}\Phi(Z)=\Phi(\nabla'_XZ)
\end{equation}
for all $X,Y\in TM,$ $Z\in \Gamma(E).$
\end{lem}
\noindent \textit{Proof of Lemma \ref{lem F isom immersion}:}
Let us consider $\Phi:TM\oplus E\rightarrow T\mathcal{H},$ $Z\mapsto \langle\langle ad(Z)\cdot\sigma(\varphi),\varphi\rangle\rangle$ and recall that the metric $B'$ on $Cl(\mathfrak{g})$ defined in Section \ref{section adjoint in spin} is invariant by left and right multiplication by $Spin(\mathfrak{g}).$ We thus obtain by (\ref{Bp ad B})
$$B'(\Phi(X),\Phi(Y))=B'([ad(X)],[ad(Y)])=B([X],[Y])=\langle X,Y\rangle$$
for all $X,Y\in TM\oplus E,$ which shows that $F$ is an isometric immersion (for $X,Y\in TM$), $\Phi$ maps $E$ to $E^F$ (for $X\in TM$ and $Y\in E$) and is an isometry (for $X,Y\in E$). We now prove (\ref{II preserve}): since $\mathcal{H}$ is totally geodesic in $\widetilde{Ad}(G),$ we have that
\begin{equation}\label{IIff df}
II^F(F_*X,F_*Y)=\{\widetilde\nabla^o_X(\partial_YF)\}^N,
\end{equation}
where $\widetilde\nabla^o$ stands for the Levi-Civita connection in $\widetilde{Ad}(G)$ and the upper-script $\{.\}^N$ means that we take the component of the vector which is normal to $F$ in $\mathcal{H}.$ Let us first note that, by (\ref{def nablao tilde}), for $X\in TM$ and $Y\in\Gamma(TM),$
\begin{equation}\label{nabla dF}
F^{-1}\widetilde\nabla^o_X(\partial_YF)=\partial_X(F^{-1}\partial_YF)+\frac{1}{2}\left[F^{-1}\partial_XF,F^{-1}\partial_YF\right]
\end{equation}
where the bracket is here the commutator in the Clifford algebra $Cl(\mathfrak{g}).$ We compute the first term in the right hand side: by  (\ref{F^-1dF}) and (\ref{killing equation}) we get 
\begin{eqnarray}
\partial_X\left(F^{-1}\partial_YF\right)&=&-\sigma\partial_X\langle\langle ad(Y)\cdot\varphi,\varphi\rangle\rangle\nonumber\\
&=&-\sigma\left\{\langle\langle \nabla_X(ad(Y))\cdot\varphi,\varphi\rangle\rangle+\langle\langle ad(Y)\cdot\nabla_X\varphi,\varphi\rangle\rangle+\langle\langle ad(Y)\cdot\varphi,\nabla_X\varphi\rangle\rangle\right\}\nonumber\\
&=&-\sigma\langle\langle\{ \nabla_X(ad(Y))+\frac{1}{2}\left(II(X)\cdot ad(Y)-ad(Y)\cdot II(X)\right)\}\cdot\varphi,\varphi\rangle\rangle\label{comp1}\\
&&-\frac{\sigma}{2}\langle\langle \left\{ad(X)\cdot ad(Y)-ad(Y)\cdot ad(X)\right\}\cdot\varphi,\varphi\rangle\rangle.\nonumber
\end{eqnarray}
The second term in the right hand side of (\ref{nabla dF}) is
\begin{equation}\label{comp2}
\frac{1}{2}\left[F^{-1}\partial_XF,F^{-1}\partial_YF\right]=\frac{\sigma}{2}\langle\langle \left\{ad(X)\cdot ad(Y)-ad(Y)\cdot ad(X)\right\}\cdot\varphi,\varphi\rangle\rangle,
\end{equation}
and cancels with the last term in (\ref{comp1}). (\ref{nabla dF}) thus implies that
\begin{equation}\label{nabla dF simplified}
\widetilde\nabla^o_X(\partial_YF)=\langle\langle\{ \nabla_X(ad(Y))+\frac{1}{2}\left(II(X)\cdot ad(Y)-ad(Y)\cdot II(X)\right)\}\cdot\sigma(\varphi),\varphi\rangle\rangle.
\end{equation}
Let us note the following formula: for $X\in TM$ and $Y\in\Gamma(TM\oplus E),$
\begin{equation}\label{general II nabla preserved}
\nabla_X(ad(Y))+\frac{1}{2}\left(II(X)\cdot ad(Y)-ad(Y)\cdot II(X)\right)=ad(\nabla_XY)+ad(\widetilde{II}(X,Y)).
\end{equation}
We postpone its proof to the end of the section. Formula (\ref{II preserve}) then follows: for $Y\in\Gamma(TM),$ $\nabla_XY$ belongs to $TM$ and the term $\langle\langle ad(\nabla_XY)\cdot\sigma(\varphi),\varphi\rangle\rangle$ is tangent to the immersion; moreover,  $\widetilde{II}(X,Y)=II(X,Y)$ in that case, and $\langle\langle ad(II(X,Y))\cdot\sigma(\varphi),\varphi\rangle\rangle$ is normal to the immersion. So (\ref{IIff df}) together with (\ref{nabla dF simplified}) and (\ref{general II nabla preserved}) yields 
\begin{eqnarray*}
II(F_*X,F_*Y)&=&\left\{\widetilde{\nabla}_X^o(\partial_YF)\right\}^N\\
&=&\langle\langle\left\{ad(\nabla_XY)+ad(II(X,Y))\right\}\cdot\sigma(\varphi),\varphi\rangle\rangle^N\\
&=&\Phi(II(X,Y)),
\end{eqnarray*}
which is (\ref{II preserve}). We similarly prove (\ref{nabla' preserve}): by definition, we have
$${\nabla'}^F_{X}\Phi(Z)=\{\widetilde\nabla^o_X(\Phi(Z))\}^N$$
with
\begin{equation}\label{nabla Phi(Z)}
F^{-1}\widetilde\nabla^o_X(\Phi(Z))=\partial_X(F^{-1}\Phi(Z))+\frac{1}{2}\left[F^{-1}\partial_XF,F^{-1}\Phi(Z)\right].
\end{equation}
Since $F^{-1}\Phi(Z)=-\sigma\langle\langle \widetilde{ad}(Z)\cdot\varphi,\varphi\rangle\rangle,$ we respectively compute the first and the second terms in the right hand side of (\ref{nabla Phi(Z)}) as in (\ref{comp1}) and (\ref{comp2}), and easily get
$$\widetilde\nabla^o_X(\Phi(Z))=\langle\langle \{\nabla_X(ad(Z))+\frac{1}{2}\left(II(X)\cdot ad(Z)-ad(Z)\cdot II(X)\right)\}\cdot\sigma(\varphi),\varphi\rangle\rangle.$$
By (\ref{general II nabla preserved}) with $Y=Z\in\Gamma(E)$ we have the formula
$$\nabla_X(ad(Z))+\frac{1}{2}\left(II(X)\cdot ad(Z)-ad(Z)\cdot II(X)\right)=ad(\nabla'_XZ)-ad(II^*(X,Z)).$$
Formula (\ref{nabla' preserve}) follows since $\langle\langle ad(\nabla'_XZ)\cdot\sigma(\varphi),\varphi\rangle\rangle$ and $\langle\langle ad(II^*(X,Z))\cdot\sigma(\varphi),\varphi\rangle\rangle$ are respectively normal and tangent to the immersion. Let us finally prove (\ref{general II nabla preserved}). By Lemmas \ref{lem1 ap1} and \ref{lem2 ap1} in the appendix, the term $\frac{1}{2}\nabla_X(ad(Y))\in Cl_{\Sigma}$ represents the endomorphism of $(TM\oplus E)^{\C}$
$$U\mapsto \nabla_X(ad(Y))(U)=\nabla_X(ad(Y)(U))-ad(Y)(\nabla_XU)$$
and the term $\frac{1}{4}\left(II(X)\cdot ad(Y)-ad(Y)\cdot II(X)\right)$ the endomorphism
$$U\mapsto \widetilde{II}(X,ad(Y)(U))-ad(Y)(\widetilde{II}(X,U)).$$
Thus, setting $\nabla_X^oT=\nabla_XT+\widetilde{II}(X,T),$ the sum represents the endomorphism
$$U\mapsto \nabla_X^o(ad(Y)(U))-ad(Y)(\nabla^o_XU);$$
this is the map $U\mapsto ad(\nabla_X^oY)(U)$ by the compatibility assumption (\ref{compatibility nablao ad}), which is represented by $\frac{1}{2}ad(\nabla^o_XY)$ in $Cl_\Sigma,$ and the result follows.

\section{Fundamental equations of Gauss, Ricci and Codazzi and a fundamental theorem in $G/H$}\label{section fundamental equations}
\subsection{Fundamental equations of Gauss, Ricci and Codazzi}\label{section fundamental equations 1}
Let us first recall the fundamental equations of the submanifold theory: for a submanifold $M$ of a Riemannian manifold $\overline{M}$,  if $\overline{R}$ denotes the curvature of $\overline{M}$ and if $R^T$ and $R^N$ denote the curvatures of the connections on $TM$ and $E,$ we have, for all $X,Y,Z\in\Gamma(TM)$ and $N\in\Gamma(E),$ 
\begin{enumerate}
\item the Gauss equation
\begin{equation}\label{gauss equation} 
(\overline{R}(X,Y)Z)^T=R^T(X,Y)Z-II^*(X,II(Y,Z))+II^*(Y,II(X,Z)),
\end{equation}
\item the Ricci equation
\begin{equation}\label{ricci equation}
(\overline{R}(X,Y)N)^N=R^N(X,Y)N-II(X,II^*(Y,N))+II(Y,II^*(X,N)),
\end{equation}
\item the Codazzi equation
\begin{equation}\label{codazzi equation}
(\overline{R}(X,Y)Z)^N=\widetilde{\nabla}_X II(Y,Z)-\widetilde{\nabla}_Y II(X,Z);
\end{equation}
\end{enumerate}
in these formulas $II:TM\times TM\rightarrow E$ is the second fundamental form of $M$ in $\overline{M},$ $II^*:TM\times E\rightarrow TM$ is such that
$$\langle II(X,Y),N\rangle=\langle Y,II^*(X,N)\rangle$$
for all $X,Y\in TM$ and $N\in E,$ and $\widetilde{\nabla}$ denotes the natural connection on $T^*M\otimes T^*M\otimes E.$
\\

Let us show that these three equations are contained in Equation (\ref{killing equation}). We assume that a solution $\varphi$ of (\ref{killing equation}) is given and we compute the curvature $\mathcal{R}$ of the spinorial connection. Let us fix a point $x_o\in M,$ and assume that $X,Y\in\Gamma(TM)$ are vector fields in the neighborhood of $x_o$ such that $\nabla X=\nabla Y=0$ at $x_o.$ Using (\ref{killing equation}) twice, we get
\begin{eqnarray*}
\nabla_X(\nabla_Y\varphi)&=&\left\{-\frac{1}{2}\nabla_XII(Y)+\frac{1}{4}II(Y)\cdot II(X)\right.\\
&&\left.+\frac{1}{4}ad(Y)\cdot ad(X)+\frac{1}{4}(II(Y)\cdot ad(X)-ad(Y)\cdot II(X)\right\}\cdot\varphi
\end{eqnarray*}
and
\begin{eqnarray}
\mathcal{R}(X,Y)\varphi&=&\nabla_X(\nabla_Y\varphi)-\nabla_Y(\nabla_X\varphi)\nonumber\\
&=&\left(\mathcal{A}+\mathcal{B}+\mathcal{C}+\mathcal{D}\right)\cdot\varphi\label{R ABCD}
\end{eqnarray}
with
$$\mathcal{A}=\frac{1}{2}\left(\nabla_YII(X)-\nabla_XII(Y)\right),$$
$$\mathcal{B}=\frac{1}{4}\left(II(Y)\cdot II(X)-II(X)\cdot II(Y)\right),$$
$$\mathcal{C}=\frac{1}{4}\left(ad(Y)\cdot ad(X)-ad(X)\cdot ad(Y)\right)$$
and
$$\mathcal{D}=\frac{1}{4}\left\{(II(Y)\cdot ad(X)-ad(X)\cdot II(Y))-(II(X)\cdot ad(Y)-ad(Y)\cdot II(X))\right\}.$$
Moreover, the left hand side of (\ref{R ABCD}) is 
$$\mathcal{R}(X,Y)\varphi=\frac{1}{2}\left(R^T(X,Y)+R^N(X,Y)\right)\cdot \varphi.$$
We may identify the Clifford coefficients in (\ref{R ABCD}) since $\varphi$ is represented by an element belonging to $Spin(\mathfrak{g})$ and thus invertible in $Cl(\mathfrak{g})$, and deduce
$$\frac{1}{2}R^T(X,Y)+\frac{1}{2}R^N(X,Y)=\mathcal{A}+\mathcal{B}+\mathcal{C}+\mathcal{D}.$$
Let us first note that the terms $R^T(X,Y), R^N(X,Y), \mathcal{A}, \mathcal{B}, \mathcal{C}$ belong to $\Lambda^2(TM\oplus E),$ whereas $\mathcal{D}$ belongs to $i\Lambda^2(TM\oplus E):$ thus $\mathcal{D}=0,$ that is $II$ and $ad$ satisfy the symmetry property
\begin{equation}\label{sym II ad}
II(X)\cdot ad(Y)-ad(Y)\cdot II(X)=II(Y)\cdot ad(X)-ad(X)\cdot II(Y)
\end{equation}
for all $X,Y\in TM.$ This identity is in fact a consequence of (\ref{compatibility II ad}) since $II$ is assumed to be symmetric. According to Lemmas \ref{lem1 ap1} and \ref{lem3 ap1} in the appendix, we then note that $\frac{1}{2}R^T(X,Y)\in\Lambda^2TM$ and $\frac{1}{2}R^N(X,Y)\in\Lambda^2E$ represent respectively the transformations 
$$Z\in TM\mapsto R^T(X,Y)Z\ \in TM \hspace{.3cm}\mbox{and}\hspace{.3cm}N\in E\mapsto R^N(X,Y)(N)\in E,$$
$\mathcal{A}\in TM\otimes E$ represents the transformation
$$Z\in TM\mapsto \widetilde{\nabla}_Y II(X,Z)-\widetilde{\nabla}_X II(Y,Z)\in E,$$
$\mathcal{B}\in\Lambda^2TM\oplus\Lambda^2E$ represents the transformation
$$Z\in TM\mapsto II^*(X,II(Y,Z))-II^*(Y,II(X,Z))\in TM$$
together with
$$N\in E\mapsto II(X,II^*(Y,N))-II(Y,II^*(X,N))\in E$$
(see also the calculations in \cite{BRZ} Lemma 4.2), and $\mathcal{C}\in \Lambda^2(TM\oplus E)$ represents the curvature $\overline{R}(X,Y)$ of the ambient manifold $G/H$ and may be decomposed into a sum of three terms
$$\mathcal{C}=\mathcal{C}^T+\mathcal{C}^N+\mathcal{C}'\ \in\ \Lambda^2TM\ \oplus\ \Lambda^2E\ \oplus\ TM\otimes E$$
representing respectively 
$$Z\mapsto (\overline{R}(X,Y)Z)^T,\hspace{.3cm} N\mapsto (\overline{R}(X,Y)N)^N\hspace{.3cm} \mbox{and}\hspace{.3cm} Z\mapsto (\overline{R}(X,Y)Z)^N.$$ 
The equations of Gauss, Ricci and Codazzi then easily follow.
\subsection{A fundamental theorem in $G/H$}
We suppose that $M,$ the vector bundle $E,$ the spinor bundles $\Sigma$ and $U\Sigma,$ the Clifford bundle $Cl_{\Sigma}$ and the map $ad\in (TM\oplus E)^*\otimes Cl_{\Sigma}$ are constructed as in Section \ref{section abstract setting}. We set, for all $X,Y\in TM,$ 
\begin{equation}\label{def Rbar}
\frac{1}{2}\overline{R}(X,Y):=\frac{1}{4}\left(ad(Y)\cdot ad(X)-ad(X)\cdot ad(Y)\right).
\end{equation}
It belongs to $\Lambda^2(TM\oplus E)\subset Cl_{\Sigma},$ and may alternatively be regarded as a map
$$Z\ \in TM\oplus E\ \mapsto\ \overline{R}(X,Y)Z\ \in TM\oplus E.$$
\begin{thm}\label{thm fundamental}
Let us assume that $II:TM\times TM\rightarrow E$ is bilinear, symmetric and satisfies the equations of Gauss, Ricci and Codazzi  (\ref{gauss equation})-(\ref{codazzi equation}) together with the compatibility condition (\ref{compatibility nablao ad}). Then there exists an isometric immersion $F:M\rightarrow \mathcal{H}$ and a bundle morphism $\Phi:E\rightarrow T\mathcal{H}$ which identifies $E$ to the normal bundle of $F$ into $\mathcal{H}$ and maps $II$ and $\nabla'$ to the second fundamental form and the normal connection of $F$ in $\mathcal{H}.$ Moreover, $F$ and $\Phi$ are unique, up to the action of an isometry of $\mathcal{H}.$ 
\end{thm}
\begin{proof}
We first observe that Equation (\ref{killing equation}) is solvable: setting 
$$\overline{\nabla}_X\varphi:=\nabla_X\varphi+\frac{1}{2}II(X)\cdot\varphi+\frac{1}{2}ad(X)\cdot\varphi$$
for all $X\in TM,$ the computations in Section \ref{section fundamental equations 1} show that the equations of Gauss, Ricci and Codazzi and (\ref{sym II ad}) (which is a consequence of (\ref{compatibility nablao ad})) are exactly the equations traducing that the curvature of $\overline{\nabla}$ is zero. If we interpret $\overline{\nabla}$ as a connection on the principal bundle $U\Sigma$ (of group $\widetilde{Ad}(G)$), we see that this is also equivalent to the existence of a section $\varphi\in \Gamma(U\Sigma)$ such that $\overline{\nabla}\varphi=0,$ i.e. of a solution of (\ref{killing equation}). Moreover the solution is unique up to the right action of the group $\widetilde{Ad}(G)$ on $U\Sigma.$ The formulas (\ref{weierstrass}) and (\ref{def Phi}) then give the immersion $F$ and the bundle morphism $\Phi;$ the proof of Theorem \ref{main result} in Section \ref{section proof main result} then proves the required properties.  Finally, if $\varphi$ is a solution of (\ref{killing equation}) and $a$ belongs to $\widetilde{Ad}(G),$ then the immersion corresponding to the solution $\varphi\cdot a$ of (\ref{killing equation}) is 
$$\langle\langle\sigma(\varphi\cdot a),\varphi\cdot a\rangle\rangle=\tau(a)\langle\langle\sigma(\varphi),\varphi\rangle\rangle\sigma(a)$$ 
whereas the identification between the bundle $E$ and the normal bundle of the immersion in $\mathcal{H}$ is
$$Z\mapsto\langle\langle Z\cdot\sigma(\varphi\cdot a),\varphi\cdot a\rangle\rangle=\tau(a)\langle\langle Z\cdot\sigma(\varphi),\varphi\rangle\rangle\sigma(a);$$
the action of $a\in\widetilde{Ad}(G)$ on a solution of (\ref{killing equation}) thus corresponds to the composition by the isometry $x\mapsto \tau(a)x\sigma(a)$ of $\mathcal{H}$ (see Section \ref{section model adjoint}), which proves the last claim in the theorem. 
\end{proof}

\section{The special case of $\HH^3=SL_2(\C)/SU(2)$}\label{section H3}
\subsection{Representation of a surface in $\HH^3$}\label{section H3 morel}
We recover here the result of Morel \cite{Mo} concerning the spinorial representation of a general surface in $\HH^3.$ 
\subsubsection{Groups and spinors using the complex quaternions} 
We consider the complex quaternions
$$\HH^{\C}=\{z_0+z_1I+z_2J+z_3K,\ z_0,z_1,z_2,z_3\in\C\}$$
where $I,J,K$ are such that
$$I^2=J^2=K^2=-1,\hspace{.5cm} IJ=K.$$
The set of the usual quaternions
$$\HH:=\{x_0+x_1I+x_2J+x_3K,\ x_0,x_1,x_2,x_3\in\R\}$$
naturally belongs to $\HH^{\C}.$ If $H$ is the complex bilinear map $\HH^{\C}\times\HH^{\C}\rightarrow\C$ such that
$$H(z,z)=z_0^2+z_1^2+z_2^2+z_ 3^2$$
for all $z\in\HH^{\C},$ the map
\begin{eqnarray*}
M_2(\C)&\rightarrow& \HH^{\C}\\
M=\left(\begin{array}{cc}a&b\\c&d\end{array}\right)&\mapsto&z_M:=\frac{1}{2}\left((a+d)+i(d-a)I+(b-c)J-i(b+c)K\right)
\end{eqnarray*}
is an isomorphism of algebras such that $\det M=H(z_M,z_M);$ it thus identifies $SL_2(\C)$ to the complex 3-sphere
$$\mathbb{S}^3_{\C}:=\{z\in\HH^{\C}:\ H(z,z)=1\}$$
and $SU(2)$ to the real 3-sphere
$$\mathbb{S}^3:=\{z\in\HH:\ H(z,z)=1\}.$$
Setting $G=\mathbb{S}^3_{\C}$ and $H=\mathbb{S}^3,$ we have $\mathfrak{g}=\mathfrak{h}\oplus\mathfrak{m}$ with
$$\mathfrak{g}=\C I\oplus \C J\oplus\C K\simeq\C^3,\hspace{.5cm}\mathfrak{h}=\R I\oplus \R J\oplus\R K\simeq \R^3\hspace{.3cm}\mbox{and}\hspace{.3cm}\mathfrak{m}=i\mathfrak{h}\simeq i\R^3.$$
Let us consider 
$$B(z,z)=-4(z_1^2+z_2^2+z_3^2)$$
(half the Killing form) for all $z=z_1I+z_2J+z_3K\in\mathfrak{g}.$ The map
\begin{eqnarray*}
\Psi:\hspace{.5cm}\mathfrak{g}\ \subset\HH^{\C}&\rightarrow& \HH^{\C}(2)\\
z&\mapsto&2\left(\begin{array}{cc}-iz&0\\0&iz\end{array}\right)
\end{eqnarray*}
satisfies the Clifford property
$$\Psi(z)^2=-B(z,z)\left(\begin{array}{cc}1&0\\0&1\end{array}\right)$$
for all $z\in\mathfrak{g}.$ It identifies the orthonormal basis $e_1=\frac{i}{2}I,$ $e_2=\frac{i}{2}J,$ $e_3=\frac{i}{2}K$ of $\mathfrak{m}$ with the following matrices 
$$e_1\simeq \left(\begin{array}{cc}I&0\\0&-I\end{array}\right),\ e_2\simeq \left(\begin{array}{cc}J&0\\0&-J\end{array}\right),\ e_3\simeq \left(\begin{array}{cc}K&0\\0&-K\end{array}\right),$$
and also
$$Cl(\mathfrak{g})\simeq\left\{\left(\begin{array}{cc}a&0\\0&b\end{array}\right),\ a,b\in\HH^{\C}\right\}$$
with
$$Cl^o(\mathfrak{g})\simeq\left\{\left(\begin{array}{cc}a&0\\0&a\end{array}\right),\ a\in\HH^{\C}\right\},\hspace{.5cm}Cl^{1}(\mathfrak{g})\simeq\left\{\left(\begin{array}{cc}a&0\\0&-a\end{array}\right),\ a\in\HH^{\C}\right\}.$$
Let us note that the operations $\sigma$ and $\tau$ on $Cl^o(\mathfrak{g})$ are given here by 
$$\sigma(a)=\overline{a_0}+\overline{a_1}I+\overline{a_2}J+\overline{a_3}K\hspace{.5cm}\mbox{and}\hspace{.5cm}\tau(a)=\overline{a}=a_0-a_1I-a_2J-a_3K$$
for all $a=a_0+a_1I+a_2J+a_3K\in \HH^{\C},$ and also that the spinor bundle $\Sigma^o:=\widetilde{Q}\times_{\rho_{\mathfrak{m}}} Cl^o(\mathfrak{g})$ splits into
\begin{equation}\label{splitting spinors n=2}
\Sigma^o=\Sigma^+\oplus\Sigma^-
\end{equation}
where $\Sigma^+$ and $\Sigma^-$ correspond to the decomposition in left ideals
$$\HH^{\C}=(\C\oplus\C J)(1-iI)\oplus(\C\oplus\C J)(1+iI).$$
\subsubsection{The spinorial representation of a surface}
The Lie brackets of the vectors of the orthonormal basis $(e_1,e_2,e_3)$ of $\mathfrak{m}$ are given by
$$[e_2,e_3]=ie_1,\ [e_3,e_1]=ie_2,\ [e_1,e_2]=ie_3,$$
which yields, for $X=x_1e_1+x_2e_2+x_3e_3,$ $x_1,x_2,x_3\in\C,$
\begin{eqnarray*}
ad(X)&=&\frac{1}{2}\sum_{i=1}^3 e_i\cdot ad(X)(e_i)\\
&=&\frac{1}{2}\sum_{i=1}^3 e_i\cdot [X,e_i]\\
&=&i(x_1e_2\cdot e_3+x_2e_3\cdot e_1+x_3e_1\cdot e_2)\\
&=&-X\cdot\omega_{\C}
\end{eqnarray*}
with $\omega_{\C}=ie_1\cdot e_2\cdot e_3.$  We now assume that $M$ is a surface in $\HH^3=SL_2(\C)/SU(2)=\mathbb{S}^3_{\C}/\mathbb{S}^3,$ represented by a spinor field $\varphi\in\Gamma(U\Sigma),$ as in Theorem \ref{main result}; we moreover suppose that $(e_1,e_2,e_3)$ is an orthonormal frame such that $e_1$ is normal and $e_2,e_3$ are tangent to $M.$ Thus
$$e_2\cdot ad(e_2)+e_3\cdot ad(e_3)=2\omega_{\C},$$
and, setting
$$D\varphi:=e_2\cdot\nabla_{e_2}\varphi+e_3\cdot\nabla_{e_3}\varphi,$$
the trace of the Killing-type equation (\ref{killing equation}) yields the Dirac equation
\begin{equation}\label{dirac phi H3}
D\varphi=\vec{H}\cdot\varphi-\omega_{\C}\cdot\varphi.
\end{equation}
Let us denote by $\Sigma M$ the usual spinor bundle on $M$ ($\dim_{\C}\Sigma M=2$), and use that there is an identification
\begin{eqnarray*}
\Sigma M &\rightarrow&\Sigma^+_{|M}\\
\psi&\mapsto&\psi^*
\end{eqnarray*}
such that $(X\cdot \psi)^*=X\cdot e_1\cdot\psi^*$ for all $X\in TM.$
\begin{prop}\label{prop rep morel}
If $\varphi\in\Gamma(U\Sigma)$ is a solution of (\ref{killing equation}), the spinor field $\psi\in\Gamma(\Sigma M)$ corresponding to $\varphi^+$ is a solution of the Dirac equation
\begin{equation}\label{eqn dirac H3}
D\psi=H\psi-\overline{\psi}
\end{equation}
such that
\begin{equation}\label{eqn normalisation H3}
\partial_X|\psi|^2=-\Re e\langle X\cdot\overline{\psi},\psi\rangle
\end{equation}
for all $X\in TM.$ Moreover $\psi$ never vanishes.
\end{prop}
Equation (\ref{eqn dirac H3}) together with (\ref{eqn normalisation H3}) form the spinorial characterization of the immersion of a surface in $\HH^3$ given by Morel in \cite{Mo}.
\begin{proof}
The Dirac equation (\ref{dirac phi H3}) implies that
$$e_2\cdot\nabla_{e_2}\psi^*+e_3\cdot\nabla_{e_3}\psi^*=He_1\cdot \psi^*-ie_1\cdot e_2\cdot e_3\cdot\psi^*$$
which gives
$$e_2\cdot e_1\cdot\nabla_{e_2}\psi^*+e_3\cdot e_1\cdot\nabla_{e_3}\psi^*=H\psi^*-ie_2\cdot e_3\cdot\psi^*$$
and
$$(e_2\cdot\nabla_{e_2}\psi+e_3\cdot\nabla_{e_3}\psi)^*=(H\psi-ie_2\cdot e_3\cdot\psi)^*,$$
that is (\ref{eqn dirac H3}) since $ie_2\cdot e_3\cdot\psi=\overline{\psi}.$ Equation (\ref{formula lem dF}) reads
$$\partial_X\langle\langle\sigma(\varphi),\varphi\rangle\rangle=\langle\langle ad(X)\cdot\sigma(\varphi),\varphi\rangle\rangle,$$
which implies that
\begin{equation}\label{pf eqn H3 1}
\partial_X\langle\langle\sigma(\varphi^+),\varphi^+\rangle\rangle=\langle\langle ad(X)\cdot\sigma(\varphi^+),\varphi^+\rangle\rangle.
\end{equation}
The correspondence between $\psi$ and $\varphi^+$ is explicitly given in coordinates by
$$[\psi]=z_1+jz_2\ \in\HH\ \mapsto\ [\varphi^+]=(z_1+iz_2J)(1-iI)\ \in\HH^{\C}.$$
Since
$$\sigma{[\varphi^+]}=(\overline{z_1}-i\overline{z_2}J)(1+iI),\hspace{.5cm}\tau[\varphi^+]=\overline{[\varphi^+]}=(1+iI)(z_1-iz_2J),$$
we easily get
\begin{equation}\label{pf eqn H3 2}
\langle\langle\sigma(\varphi^+),\varphi^+\rangle\rangle=\overline{[\varphi^+]}\sigma{[\varphi^+]}=2(|z_1|^2+|z_2|^2)(1+iI)
\end{equation}
and also
\begin{eqnarray}
\langle\langle ad(X)\cdot\sigma(\varphi^+),\varphi^+\rangle\rangle&=&\overline{[\varphi^+]}[ad(X)]\sigma{[\varphi^+]}\nonumber\\
&=&(1+iI)(z_1-iz_2J)(ix_2J+ix_3K)(\overline{z_1}-i\overline{z_2}J)(1+iI)\nonumber\\
&=&-4\Re e\left\{(x_2-ix_3)z_1\overline{z_2}\right\}(1+iI).\label{pf eqn H3 3}
\end{eqnarray}
Moreover, since $[\psi]=z_1+jz_2,$ we have
$$|\psi|^2=|z_1|^2+|z_2|^2\hspace{.5cm}\mbox{and}\hspace{.5cm}\Re e\langle X\cdot\overline{\psi},\psi\rangle=2\Re e\left\{(x_2-ix_3)z_1\overline{z_2}\right\}.$$
Equations (\ref{pf eqn H3 1}), (\ref{pf eqn H3 2}) and (\ref{pf eqn H3 3}) give (\ref{eqn normalisation H3}). Finally, $\psi\simeq\varphi^+$ does not vanish since 
$$H([\varphi],[\varphi])=2H([\varphi^+],[\varphi^-])=1\neq 0.$$
\end{proof}
\begin{rem}\label{rmk rep H3}
i) The explicit representation formula gives here an immersion in the hyperboloid model of $\HH^3$ in the four-dimensional Minkowski space $\R^{1,3}:$ for
$$[\varphi]=z_0+z_1I+z_2J+z_3K\ \in \mathbb{S}^3_{\C},$$
we get by a direct computation
\begin{eqnarray*}
F&=&\tau[\varphi]\sigma{[\varphi]}\\
&=&X_0+iX_1I+iX_2J+iX_3K\hspace{.3cm}\in\ \R\oplus i\R I\oplus i\R J\oplus i\R K
\end{eqnarray*}
with
$$X_0=|z_0|^2+|z_1|^2+|z_2|^2+|z_3|^2,\hspace{.3cm} iX_1=z_0\overline{z_1}-z_1\overline{z_0}-z_2\overline{z_3}+z_3\overline{z_2},$$
$$iX_2=z_0\overline{z_2}-z_2\overline{z_0}+z_1\overline{z_3}-z_3\overline{z_1},\hspace{.3cm} iX_3=z_0\overline{z_3}-z_3\overline{z_0}-z_1\overline{z_2}+z_2\overline{z_1}.$$
We have $X_0>0$ and 
$$-X_0^2+X_1^2+X_2^2+X_3^2=-1$$
as a direct consequence of $H([\varphi],[\varphi])=1.$
\\
\\ii) Theorem \ref{thm fundamental} is here the usual fundamental theorem in $\HH^3.$ Indeed, the various assumptions made in Section \ref{section abstract setting} are satisfied: firstly, the existence of a reduction $\widetilde{Q}_H$ of the bundle $\widetilde{Q}$ is not an additional requirement since $H=SU(2)=Spin(\mathfrak{m})$ for $\mathfrak{m}=\R^3;$ secondly, (\ref{compatibility nablao ad}) or equivalently (\ref{compatibility II ad}) and (\ref{compatibility nabla ad}) are satisfied. Let us first see that (\ref{compatibility II ad}) holds: let us denote by $S:TM\rightarrow TM$ the symmetric operator such that $II(X,Y)=\langle S(X),Y\rangle e_1$ for all $X,Y\in TM;$ we have
\begin{eqnarray*}
\frac{1}{2}II(X)&=&\frac{1}{2}\left(e_2\cdot II(X,e_2)+e_3\cdot II(X,e_3)\right)\\
&=&\frac{1}{2}\left(\langle S(X),e_2\rangle e_2+\langle S(X),e_3\rangle e_3\right)\cdot e_1\\
&=&\frac{1}{2}S(X)\cdot e_1
\end{eqnarray*}
and 
\begin{eqnarray*}
\frac{1}{2}II(X)\cdot \frac{1}{2}ad(Y)-\frac{1}{2}ad(Y)\cdot \frac{1}{2}II(X)&=&\frac{1}{4}\left\{S(X)\cdot e_1\cdot(-Y\cdot\omega_{\C})-(-Y\cdot\omega_{\C})\cdot S(X)\cdot e_1 \right\}\\
&=&\frac{1}{4}(S(X)\cdot Y+Y\cdot S(X))\cdot e_1 \cdot\omega_{\C}\\
&=&-\frac{1}{2}\langle S(X),Y\rangle\ e_1\cdot\omega_{\C}\\
&=&\frac{1}{2}ad(II(X,Y))
\end{eqnarray*}
which is (\ref{compatibility II ad}) written in the Clifford bundle. We now show that  (\ref{compatibility nabla ad}) holds, i.e. that
$$\nabla_X(ad(Y)(Z))=ad(Y)(\nabla_XZ)+ad(\nabla_XY)(Z)$$
for all $X\in TM$ and $Y,Z\in\Gamma(TM\oplus E),$ or equivalently, in the Clifford bundle,
$$\nabla_X(ad(Y)\cdot Z-Z\cdot ad(Y))=ad(Y)\cdot\nabla_XZ-\nabla_XZ\cdot ad(Y)+ad(\nabla_XY)\cdot Z-Z\cdot ad(\nabla_XY).$$
But this amounts to show that $\nabla_X(ad(Y))=ad(\nabla_XY),$ that is $\nabla_X(Y\cdot \omega_{\C})=\nabla_XY\cdot\omega_{\C},$ which is evident since $\omega_{\C}=ie_1\cdot e_2\cdot e_3$ and $\nabla\omega_{\C}=0$ (since $\nabla e_1=0$ and $\nabla (e_2\cdot e_3)=0$). Finally, the tensor $\overline{R}$ in (\ref{def Rbar}) represents the curvature tensor $R^0$ of $\HH^3:$ we have 
$$\frac{1}{2}\overline{R}(X,Y)=\frac{1}{4}\left(ad(Y)\cdot ad(X)-ad(X)\cdot ad(Y)\right)=\frac{1}{4}\left(X\cdot Y-Y\cdot X\right)$$
and thus
\begin{eqnarray*}
R^0(X,Y)Z&:=&-\langle Y,Z\rangle X+\langle X,Z\rangle Y\\
&=&-\frac{1}{2}\langle Y,Z\rangle X+\frac{1}{2}\langle X,Z\rangle Y-\frac{1}{2}X\langle Y,Z\rangle +\frac{1}{2}Y\langle X,Z\rangle\\
&=&\frac{1}{4}\left\{(Y\cdot Z+Z\cdot Y)\cdot X-(X\cdot Z+Z\cdot X)\cdot Y\right.\\
&&\left.+X\cdot (Y\cdot Z+Z\cdot Y)-Y\cdot (X\cdot Z+Z\cdot X)\right\}\\
&=&\frac{1}{4}\left\{Z\cdot Y\cdot X-Z\cdot X\cdot Y+X\cdot Y\cdot Z-Y\cdot X\cdot Z\right\}\\
&=&\frac{1}{2}\overline{R}(X,Y)\cdot Z-Z\cdot \frac{1}{2}\overline{R}(X,Y),
 \end{eqnarray*}
 which implies the result by Lemma \ref{lem1 ap1} in the appendix.
 \end{rem}
\subsection{Weierstrass-type representation of surfaces with constant mean curvature $1$ in $\HH^3$}\label{section H3 CMC1}
We suppose that $M$ is a surface in $\HH^3,$ with constant mean curvature $1,$ represented by a spinor field $\varphi,$ as in Theorem \ref{main result} and the previous section. We choose a conformal parameter $z=x+iy$ of $M:$ the metric reads $\displaystyle{\lambda^2(dx^2+dy^2)}$ for some positive function $\lambda.$ We fix a unit section $e_1$ of the trivial line bundle $E,$ and consider the component $g=[\varphi]\in \mathbb{S}^3_{\C}$ of the spinor field $\varphi$ in a spinorial frame above the frame $(e_1,\frac{1}{\lambda}\partial_x,\frac{1}{\lambda}\partial_y).$ Since the surface has constant mean curvature 1, the matrix of the real second fundamental form $\langle II,e_1\rangle$ in $(\frac{1}{\lambda}\partial_x,\frac{1}{\lambda}\partial_y)$ is of the form
$$\left(\begin{array}{cc}\alpha+1&\gamma\\\gamma&1-\alpha\end{array}\right)$$
for some real functions $\alpha,\gamma.$ By a direct computation, Equation (\ref{killing equation}) reads
\begin{equation}\label{eqn bryant 1}
dg\ g^{-1}=\eta+\frac{i}{2}\lambda dz J(1+iI)
\end{equation}
where
\begin{equation}\label{eqn bryant 2}
\eta=\frac{1}{2\lambda}\left(\partial_y\lambda\ dx-\partial_x\lambda\ dy\right)I-\frac{\lambda}{2}\left\{\left(\gamma\ dx-\alpha\ dy\right)J-\left(\alpha\ dx+\gamma\ dy\right)K\right\};
\end{equation}
the first term in the right hand side of (\ref{eqn bryant 2}) represents the Levi-Civita connection 1-form  in $(\frac{1}{\lambda}\partial_x,\frac{1}{\lambda}\partial_y),$ and the second term represents the traceless part of the second fundamental form; the last term in (\ref{eqn bryant 1}) is the sum of the contributions of $ad(X)$ and of the trace of the second fundamental form (we use here that the mean curvature is constant equal to $1$). If $h: M\rightarrow \mathbb{S}^3$ is a solution of
\begin{equation}\label{eqn bryant 3}
d\overline{h}\ h=\eta
\end{equation}
(this equation is solvable in $\mathbb{S}^3$ since $\eta$ has real coefficients and satisfies the structure equation $d\eta-\eta\wedge\eta=0$ (by a computation using (\ref{eqn bryant 2}), or by Remark \ref{rem ngeq3} below)), it directly follows from (\ref{eqn bryant 1}) that $v=hg$ is a solution of
\begin{equation}\label{eqn bryant 4}
dv\ v^{-1}=\frac{i}{2}\lambda dz\ h J(1+iI)\overline{h}.
\end{equation}
This implies that $dv$ is $\C$-linear, and thus that $v$ is holomorphic. This also implies that $H(dv,dv)=0.$ Moreover, the immersion is explicitly given by
\begin{equation}\label{eqn bryant 5}
F=\langle\langle\sigma(\varphi),\varphi\rangle\rangle=g^{-1}\ \sigma(g)=v^{-1}\ \sigma(v),
\end{equation}
since $g=h^{-1}v$ with $h \sigma(h^{-1})=1$ ($h$ belongs to $\mathbb{S}^3,$ which reads $\sigma(h)=h$). The explicit representation formula (\ref{eqn bryant 5}) where $v$ is holomorphic and such that $H(dv,dv)=0$ is the Bryant representation of the surfaces with constant mean curvature $1$ in $\HH^3.$
\begin{rem}
The function $h:M\rightarrow \mathbb{S}^3$ may be interpreted as the component of a spinor field $\psi\in\Sigma M$ in a spinorial frame above $(\frac{1}{\lambda}\partial_x,\frac{1}{\lambda}\partial_y).$ Equations (\ref{eqn bryant 2}) and (\ref{eqn bryant 3}) show that $\psi$ represents an isometric immersion of $M$ in $\R^3;$ this is a minimal immersion, since the second fundamental form of the immersion in $\R^3$ is given by the traceless part of the second fundamental form of the immersion in $\HH^3$ (see the expression (\ref{eqn bryant 2}) of $\eta$ appearing in (\ref{eqn bryant 3})). The transformation $\varphi\mapsto\psi$ corresponds to a Lawson-type correspondence between surfaces with constant mean curvature $1$ in $\HH^3$ and minimal surfaces in $\R^3.$
\end{rem}
\section{The symmetric space $SL_n(\C)/SU(n),$ $n\geq 3$}\label{section n geq 3}
We first write Equation (\ref{killing equation}) in another form: if $z=x+iy$ is a conformal parameter of $M,$ the 1-form $ad\in\Gamma(T^*M\otimes Cl_{\Sigma})$ may be written
\begin{eqnarray*}
ad&=&ad(\partial_z)dz+ad(\partial_{\overline{z}})d\overline{z}\\
&=&2ad(\partial_z)dz+\left\{ad(\partial_{\overline{z}})d\overline{z}-ad(\partial_z)dz\right\}
\end{eqnarray*}
where the last term reads
$$ad(\partial_{\overline{z}})d\overline{z}-ad(\partial_z)dz=iad(-dy\partial_x+dx\partial_y)=i ad\circ J$$
where $J:TM\rightarrow TM$ is the natural complex structure and is therefore a 1-form with values in $\Lambda^2(TM\oplus E)\subset Cl_{\Sigma}$ (recall that $ad(X)$ maps $TM\oplus E$ to $i(TM\oplus E),$ see Section \ref{section abstract setting}); if we consider the connection
\begin{equation}
\overline{\nabla}':=\nabla+\frac{1}{2}II+\frac{i}{2}ad\circ J,
\end{equation}
where $II$ is also regarded here as a 1-form with values in $\Lambda^2(TM\oplus E)\subset Cl_{\Sigma}$, Equation (\ref{killing equation}) may thus be written in the form 
\begin{equation}\label{nablap ad}
\overline{\nabla}'\varphi=-ad(\partial_z)dz\cdot\varphi.
\end{equation}

The connection $\overline{\nabla}'$ is a connection on $\Sigma,$ and may also be considered as a connection on the $Spin(\mathfrak{m})$ principal bundle 
$$\widetilde{Q}=\widetilde{Q}_{p,q}\times_{r_{p,q}} Spin(\mathfrak{m})$$ 
constructed from the $Spin(p)\times Spin(q)$ principal bundle $\widetilde{Q}_{p,q}$ and the morphism 
$$r_{p,q}:\ Spin(p)\times Spin(q)\rightarrow Spin(\mathfrak{m}).$$
Let us recall that $F$ is an immersion of $M$ into $Spin(\mathfrak{g});$ the right Gauss map is defined here by 
$$\nu_R:=\left[\partial_zF\ F^{-1}\right]\hspace{.3cm}\in\mathbb{P}(\Lambda^2\mathfrak{g})$$
where $\mathbb{P}(\Lambda^2\mathfrak{g})$ is the complex projective space of $\Lambda^2\mathfrak{g}$ (it is similar to the definition in \cite{KTUY}).
\begin{prop}\label{prop nuR holo}
The right Gauss map is holomorphic if and only if the connection $\overline{\nabla}'$ is flat.
\end{prop}
This relies on the following lemma:
\begin{lem}\label{lem nuR holo}
$\nu_R$ is holomorphic if and only so is $\langle\langle ad(\partial_z)\cdot\varphi,\varphi\rangle\rangle.$ 
\end{lem}
\begin{proof}
We adapt in our context ideas of \cite{KTUY}. Let us first note that, by Lemma \ref{lem F-1dF},
$$\partial_zF\ F^{-1}=\langle\langle ad(\partial_z)\cdot\varphi,\varphi\rangle\rangle\hspace{.3cm} \in\Lambda^2\mathfrak{g},$$
which shows that $\nu_R$ is holomorphic if so is $\langle\langle ad(\partial_z)\cdot\varphi,\varphi\rangle\rangle.$ We now assume that $\nu_R$ is holomorphic: there exist an holomorphic function $\Phi:M\rightarrow\Lambda^2\mathfrak{g}$ and a smooth function $\mu:M\rightarrow\C^*$ such that
\begin{equation}\label{def mu}
\partial_zF\ F^{-1}=\mu\Phi.
\end{equation}
Since
$$\nabla_{\partial_{\overline{z}}}\varphi=-\frac{1}{2}II(\partial_{\overline{z}})\cdot\varphi-\frac{1}{2}ad(\partial_{\overline{z}}),$$
we have that
\begin{eqnarray}
\partial_{\overline{z}}\langle\langle ad(\partial_z)\cdot\varphi,\varphi\rangle\rangle&=&\langle\langle\nabla_{\partial_{\overline{z}}}(ad(\partial z))\cdot\varphi,\varphi\rangle\rangle\label{nuR holo interm1}\\
&&+\frac{1}{2}\langle\langle\left\{II(\partial_{\overline{z}})\cdot ad(\partial_z)-ad(\partial_z)\cdot II(\partial_{\overline{z}})\right\}\cdot\varphi,\varphi\rangle\rangle\nonumber\\
&&+\frac{1}{2}\langle\langle\left\{ad(\partial_{\overline{z}})\cdot ad(\partial_z)-ad(\partial_z)\cdot ad(\partial_{\overline{z}})\right\}\cdot\varphi,\varphi\rangle\rangle.\nonumber
\end{eqnarray}
But formula (\ref{general II nabla preserved}) readily gives that
\begin{equation}\label{nuR holo interm2}
\nabla_{\partial_{\overline{z}}}(ad(\partial z))+\frac{1}{2}\left\{II(\partial_{\overline{z}})\cdot ad(\partial_z)-ad(\partial_z)\cdot II(\partial_{\overline{z}})\right\}=ad(\nabla^o_{\partial_{\overline{z}}}\partial z)
\end{equation}
where we have set
$$\nabla^o_{\partial_{\overline{z}}}\partial z:=\nabla_{\partial_{\overline{z}}}\partial z+II(\partial_{\overline{z}},\partial_z).$$
Since we also have by (\ref{def mu})
$$\partial_{\overline{z}}\langle\langle ad(\partial_z)\cdot\varphi,\varphi\rangle\rangle=\partial_{\overline{z}}(\log\mu)\ \langle\langle ad(\partial_z)\cdot\varphi,\varphi\rangle\rangle,$$
we deduce from (\ref{nuR holo interm1}) and (\ref{nuR holo interm2}) that
$$\partial_{\overline{z}}(\log\mu)\ ad(\partial_z)=ad(\nabla^o_{\partial_{\overline{z}}}\partial z)+\frac{1}{2}\left\{ad(\partial_{\overline{z}})\cdot ad(\partial_z)-ad(\partial_z)\cdot ad(\partial_{\overline{z}})\right\}.$$
But the right hand side is invariant by $*=\sigma\circ \tau,$ and so is the left hand side; we deduce that
$$\overline{\partial_{\overline{z}}(\log\mu)}\ ad(\partial_{\overline{z}})=\partial_{\overline{z}}(\log\mu)\ ad(\partial_z),$$
which implies that
\begin{eqnarray*}
\partial_{\overline{z}}(\log\mu)B(ad(\partial_z),ad(\partial_{\overline{z}}))&=&B(\partial_{\overline{z}}(\log\mu)\ ad(\partial_z),ad(\partial_{\overline{z}}))\\
&=&B(\overline{\partial_{\overline{z}}(\log\mu)}\ ad(\partial_{\overline{z}}),ad(\partial_{\overline{z}}))\\
&=&\overline{\partial_{\overline{z}}(\log\mu)}\ B(ad(\partial_{\overline{z}}),ad(\partial_{\overline{z}}))\\
&=&0
\end{eqnarray*}
since, by (\ref{Bp ad B}) and $B'=-2\lambda B$, we have
\begin{equation}\label{B ad zb 0}
-2\lambda B(ad(\partial_{\overline{z}}),ad(\partial_{\overline{z}}))=B(\partial_{\overline{z}},\partial_{\overline{z}})=B(\partial_x,\partial_x)-B(\partial_y,\partial_y)+2iB(\partial_x,\partial_y)=0
\end{equation}
(in the parameter $z=x+iy$ the metric is conformal). This in turn implies that $\partial_{\overline{z}}(\log\mu)=0$ since $B(ad(\partial_z),ad(\partial_{\overline{z}}))<0$ as a consequence of the computation
$$-2\lambda B(ad(\partial_z),ad(\partial_{\overline{z}}))=B(\partial_z,\partial_{\overline{z}})=B(\partial_x,\partial_x)+B(\partial_y,\partial_y)>0.$$
The function $\mu$ is thus holomorphic and the result follows.
\end{proof}
\noindent\textit{Proof of Proposition \ref{prop nuR holo}:}
By (\ref{nablap ad}), we have $\overline{\nabla}'_{\partial_{z}}\varphi=-ad(\partial_z)\cdot\varphi$ and $\overline{\nabla}'_{\partial_{\overline{z}}}\varphi=0,$ and thus
$$R(\partial_z,\partial_{\overline{z}})\varphi=-\overline{\nabla}'_{\partial_{\overline{z}}}\left(ad(\partial_z)\cdot\varphi\right).$$
Using again that $\overline{\nabla}'_{\partial_{\overline{z}}}\varphi=0,$ we deduce that
$$\partial_{\overline{z}}\langle\langle ad(\partial_z)\cdot\varphi,\varphi\rangle\rangle=\langle\langle\overline{\nabla}'_{\partial_{\overline{z}}}\left(ad(\partial_z)\cdot\varphi\right),\varphi\rangle\rangle=-\langle\langle R(\partial_z,\partial_{\overline{z}})\varphi,\varphi\rangle\rangle.$$
But the latter is zero if and only if $R=0:$ indeed, $R\varphi=0$ for the special section $\varphi$ if and only if $R\psi=0$ for all section $\psi\in\Gamma(\Sigma),$ since, if $R\varphi=0$ then $R(\varphi\cdot a)=0$ for all $a\in Cl(\mathfrak{g}).$
\begin{flushright}
$\Box$
\end{flushright}
We assume now that $\nu_R$ is holomorphic, and consider a spinorial frame $\tilde{s}\in \widetilde{Q}$ above $(\frac{1}{\lambda}\partial_x,\frac{1}{\lambda}\partial_y)$ together with $g\in Spin(\mathfrak{g})$ such that $\varphi=[\tilde{s},g].$ In $\tilde{s},$ (\ref{nablap ad}) reads
\begin{equation}\label{dg eta w}
dg\ g^{-1}=\eta+wdz
\end{equation}
where $\eta\in\Omega^1(M,\Lambda^2\mathfrak{m})$ and $w:M\rightarrow\Lambda^2\mathfrak{g}$ respectively represent the connection form of $\overline{\nabla}'$ and $-ad(\partial_z)$ in $\tilde{s}.$ If we consider a parallel section $\tilde{s}'=\tilde{s}.h^{-1}$ of $\widetilde{Q}$ and $v=hg\in Spin(\mathfrak{g})$ representing $\varphi$ in $\tilde{s}',$ (\ref{dg eta w}) simplifies to 
\begin{equation}\label{eqn bryant 6}
dv\ v^{-1}=dz\ h w\overline{h},
\end{equation}
which implies that $v$ is holomorphic and $B(dv,dv)=0$ (since $B(ad(\partial_z),ad(\partial_z))=0,$ as in (\ref{B ad zb 0})). Moreover, the immersion is explicitly given by
\begin{equation}\label{eqn bryant 7}
F=\langle\langle\sigma(\varphi),\varphi\rangle\rangle=g^{-1}\ \sigma(g)=v^{-1}\ \sigma(v).
\end{equation}
This is essentially the generalized Weierstrass-Bryant representation formula given in \cite{KTUY}.
\begin{rem}\label{rem ngeq3}
The function $v$ may also be constructed by a direct computational argument: the integrability of (\ref{dg eta w}) implies that
$$\left(d\eta-\eta\wedge\eta\right)(\partial_z,\partial_{\overline{z}})=\partial_{\overline{z}}w+w\eta(\partial_{\overline{z}})-\eta(\partial_{\overline{z}})w.$$
Now, since by (\ref{dg eta w}) $\partial_{\overline{z}}g=\eta(\partial_{\overline{z}})\ g$ and $\partial_{\overline{z}}\ g^{-1}=-g^{-1}\eta(\partial_{\overline{z}}),$ we get
\begin{eqnarray*}
\partial_{\overline{z}}(g^{-1}wg)&=&g^{-1}\left(\partial_{\overline{z}}w+w\eta(\partial_{\overline{z}})-\eta(\partial_{\overline{z}})w\right)g.
\end{eqnarray*}
Since $g^{-1}wg=-\langle\langle ad(\partial_z)\cdot\varphi,\varphi\rangle\rangle$ is holomorphic by Lemma \ref{lem nuR holo}, we deduce that
$$\partial_{\overline{z}}w+w\eta(\partial_{\overline{z}})-\eta(\partial_{\overline{z}})w=0.$$
We thus obtain that $d\eta-\eta\wedge\eta=0,$ and since $\eta$ takes values in $\Lambda^2\mathfrak{m}$, we may thus consider a solution $h\in Spin(\mathfrak{m})$ of $dh\ h^{-1}=-\eta.$ It is then straightforward to check from (\ref{dg eta w}) that $v=hg$ satisfies (\ref{eqn bryant 6}).
\end{rem}
\appendix
\section{Skew-symmetric operators and Clifford algebra}
We gather here results concerning the representation of skew-symmetric operators using the Clifford algebra. These results first appeared in \cite{BRZ}, but since we use here other conventions we prefer to include the proofs. We consider $\R^n$ endowed with its canonical scalar product. A skew-symmetric operator $u:\R^n\rightarrow\R^n$ naturally identifies to a bivector in $\Lambda^2\R^n,$ and thus also to an element of the Clifford algebra $Cl_n(\R).$ We precise here this identification and the relation between the Clifford product in $Cl_n(\R)$ and the composition of endomorphisms. If $\eta$ and $\eta'$ belong to the Clifford algebra $Cl_n(\R),$ we set 
$$[\eta,\eta']=\eta\cdot \eta'-\eta'\cdot \eta,$$ 
where the dot $\cdot$ is the Clifford product. We denote by $(e_1,\ldots,e_n)$ the canonical basis of $\R^n.$
\begin{lem} \label{lem1 ap1}
Let $u:\R^n\rightarrow\R^n$ be a skew-symmetric operator. Then the bivector 
\begin{equation}\label{biv rep u}
\underline{u}=\frac{1}{4}\sum_{j=1}^ne_j\cdot u(e_j)\hspace{.3cm}\in\ \Lambda^2\R^n\subset Cl_n(\R)
\end{equation}
represents $u,$ and, for all $\xi\in\R^n,$ $[\underline{u},\xi]=u(\xi).$ In the paper, and for sake of simplicity, we will denote $\underline{u}$ by $\frac{1}{2}u.$ 
\end{lem}

\begin{proof}
For $i<j,$ we consider  the linear map 
$$u:\hspace{.5cm}e_i\mapsto e_j,\hspace{.5cm}e_j\mapsto -e_i,\hspace{.5cm} e_k\mapsto 0\hspace{.3cm}\mbox{if}\hspace{.3cm} k\neq i,j;$$
it is skew-symmetric and corresponds to the bivector $e_i\wedge e_j\in\Lambda^2\R^n;$ it is thus naturally represented by  $\underline{u}=\frac{1}{2}e_i\cdot e_j=\frac{1}{4}\left(e_i\cdot e_j-e_j\cdot e_i\right),$ which is (\ref{biv rep u}). We then compute, for $k=1,\ldots,n,$
$$[\underline{u},e_k]=\frac{1}{2}\left(e_i\cdot e_j\cdot e_k-e_k\cdot e_i\cdot e_j\right)$$
and easily get
$$[\underline{u},e_k]=e_j\hspace{.3cm}\mbox{if}\hspace{.3cm}k=i,\hspace{.3cm}-e_i\hspace{.3cm}\mbox{if}\hspace{.3cm}k=j,\hspace{.3cm}0\hspace{.3cm}\mbox{if}\hspace{.3cm}k\neq i,j.$$
The result follows by linearity.
\end{proof}

\begin{lem}\label{lem2 ap1}
Let $u:\R^n\rightarrow\R^n$ and $v:\R^n\rightarrow\R^n$ be two skew-symmetric operators, represented in $Cl_n(\R)$ by 
$$\underline{u}=\frac{1}{4}\sum_{j=1}^ne_j\cdot u(e_j)\hspace{.5cm}\mbox{and}\hspace{.5cm}\underline{v}=\frac{1}{4}\sum_{j=1}^ne_j\cdot v(e_j)$$
respectively. Then $[\underline{u},\underline{v}]\in\Lambda^2\R^n\subset Cl_n(\R)$ represents $u\circ v-v\circ u.$
\end{lem}
\begin{proof}
For $\xi\in\R^n,$ the Jacobi equation yields
$$[[\underline{u},\underline{v}],\xi]=[\underline{u},[\underline{v},\xi]]-[\underline{v},[\underline{u},\xi]].$$
Thus, using Lemma \ref{lem1 ap1} repeatedly, $[\underline{u},\underline{v}]$ represents the map 
\begin{eqnarray*}
\xi&\mapsto&[[\underline{u},\underline{v}],\xi]
\begin{array}[t]{cc}=&[\underline{u},[\underline{v},\xi]]-[\underline{v},[\underline{u},\xi]]\\
=&[\underline{u},v(\xi)]-[\underline{v},u(\xi)]\\
=&(u\circ v-v\circ u)(\xi),
\end{array}
\end{eqnarray*}
and the result follows.
\end{proof}
We now assume that $\R^n=\R^p\oplus\R^q,$ $p+q=n.$
\begin{lem}\label{lem3 ap1}
Let us consider a linear map $u:\R^p\rightarrow\R^q$ and its adjoint $u^*:\R^q\rightarrow\R^p.$ Then the bivector
$$\underline{u}=\frac{1}{2}\sum_{j=1}^pe_j\cdot u(e_j)\hspace{.3cm}\in\ \Lambda^2\R^n\subset Cl_n(\R)$$
represents 
$$\left(\begin{array}{cc}0&-u^*\\u&0\end{array}\right):\hspace{.5cm}\R^p\oplus\R^q\rightarrow\R^p\oplus\R^q,$$
we have
\begin{equation}\label{biv u u*}
\underline{u}=\frac{1}{4}\left(\sum_{j=1}^pe_j\cdot u(e_j)+\sum_{j=p+1}^ne_j\cdot (-u^*(e_j))\right)
\end{equation}
and, for all $\xi=\xi_p+\xi_q\in\R^n,$ $[\underline{u},\xi]=u(\xi_p)-u^*(\xi_q).$ As above, we will simply denote $\underline{u}$ by $\frac{1}{2}u.$
\end{lem}
\begin{proof} In view of Lemma \ref{lem1 ap1}, $\underline{u}$ represents the linear map $\xi\mapsto[\underline{u},\xi].$ We compute, for $\xi\in\R^p,$
\begin{eqnarray*}
[\underline{u},\xi]&=&\frac{1}{2}\left(\sum_{j=1}^pe_j\cdot u(e_j)\cdot\xi-\xi\cdot\sum_{j=1}^pe_j\cdot u(e_j)\right)\\
&=&-\frac{1}{2}\sum_{j=1}^p(e_j\cdot\xi+\xi\cdot e_j)\cdot u(e_j)\\
&=&\sum_{j=1}^p\langle \xi,e_j\rangle\ u(e_j)\\
&=&u(\xi),
\end{eqnarray*}
and, for $\xi\in\R^q,$
\begin{eqnarray*}
[\underline{u},\xi]&=&\frac{1}{2}\left(\sum_{j=1}^pe_j\cdot u(e_j)\cdot\xi-\xi\cdot\sum_{j=1}^pe_j\cdot u(e_j)\right)\\
&=&\frac{1}{2}\sum_{j=1}^pe_j\cdot\left(u(e_j)\cdot\xi+\xi\cdot u(e_j)\right)\\
&=&-\sum_{j=1}^pe_j\ \langle u(e_j),\xi\rangle\\
&=&-\sum_{j=1}^pe_j\ \langle e_j,u^*(\xi)\rangle\\
&=&-u^*(\xi).
\end{eqnarray*}
Finally, 
$$\underline{u}=\frac{1}{2}\sum_{j=1}^pe_j\cdot u(e_j)=\frac{1}{4}\left(\sum_{j=1}^pe_j\cdot u(e_j)+\sum_{j=1}^p-u(e_j)\cdot e_j\right)$$
with
\begin{eqnarray*}
\sum_{j=1}^p-u(e_j)\cdot e_j&=&-\sum_{i=p+1}^{p+q}\sum_{j=1}^p\langle u(e_j),e_i\rangle\ e_i\cdot e_j\\
&=&\sum_{i=p+1}^{p+q}e_i\cdot\left(-\sum_{j=1}^p\langle e_j,u^*(e_i)\rangle\ e_j\right)\\
&=&\sum_{i=p+1}^{p+q}e_i\cdot (-u^*(e_i)),
\end{eqnarray*}
which gives (\ref{biv u u*}).
\end{proof}
\section{The metric on $Cl(\mathfrak{g})$}\label{appendix B}
Let $\mathfrak{g}$ be a complex semi-simple Lie algebra of complex dimension $n,$ $B:\mathfrak{g}\times \mathfrak{g}\rightarrow\C$ a multiple of its Killing form and $Cl(\mathfrak{g})$ the associated complex Clifford algebra, as in Section \ref{section clifford spin}. The form $B$ naturally extends to a bilinear and symmetric form $B:Cl(\mathfrak{g})\times Cl(\mathfrak{g})\rightarrow\C$ since $Cl(\mathfrak{g})$ naturally identifies to $\oplus_{p=0}^n\Lambda^p\mathfrak{g}:$ the extension of $B$ is such that, if $e_1,\ldots,e_n$ is a complex basis of $\mathfrak{g}$ formed by unit and orthogonal vectors, i.e. such that $B(e_i,e_j)=\delta_{ij}$, the $p$-vectors $e_{i_1}\wedge \ldots\wedge e_{i_p}$ also form a complex basis of $\Lambda^p\mathfrak{g}$ of unit and orthogonal vectors, and $B(\Lambda^p\mathfrak{g},\Lambda^q\mathfrak{g})=0$ if $p\neq q.$
\begin{lem}\label{lem1 ap2}
The form $B$ is invariant under the left or right action of $Spin(\mathfrak{g})$ by multiplication: for all $\eta,\eta'\in Cl(\mathfrak{g})$ and $g\in Spin(\mathfrak{g}),$
\begin{equation}\label{B Clg invariant}
B(g\cdot\eta,g\cdot\eta')=B(\eta\cdot g,\eta'\cdot g)=B(\eta,\eta').
\end{equation}
It is thus also $Ad$-invariant, and satisfies
\begin{equation}\label{B Cl ad inv}
B([Z,X],Y)=-B(X,[Z,Y])
\end{equation}
for all $X,Y,Z\in\Lambda^2\mathfrak{g}\subset Cl(\mathfrak{g}),$ where the bracket is here the commutator in $Cl(\mathfrak{g}).$
\end{lem}
\begin{proof}
Let us note that $B(\eta,\eta')$ is, up to sign, the coefficient of $1_{Cl(\mathfrak{g})}$ in the product $\tau\eta'\cdot \eta\in Cl(\mathfrak{g}):$ indeed, if $H(\eta,\eta')$ denotes this coefficient, it defines a bilinear and symmetric map $H:Cl(\mathfrak{g})\times Cl(\mathfrak{g})\rightarrow\C$ such that  
$$H(e_{i_1}\wedge \ldots\wedge e_{i_p},e_{j_1}\wedge \ldots\wedge e_{j_q})=\epsilon$$
with $\epsilon=(-1)^p$ if $p=q$ and $i_1,\ldots,i_p=j_1,\ldots,j_p$, and $\epsilon=0$ otherwise; thus $B=H$ on $Cl^0(\mathfrak{g})$ and $B=-H$ on $Cl^1(\mathfrak{g}).$ Since $H$ clearly satisfies (\ref{B Clg invariant}), so does $B.$ For the last claim, if $g(t)\in Spin(\mathfrak{g})$ is a curve such that $g(0)=1_{Cl(\mathfrak{g})}$ and $g'(0)=Z,$ (\ref{B Clg invariant}) implies that
$$B(g(t)\cdot X\cdot g(t)^{-1},g(t)\cdot Y\cdot g(t)^{-1})=B(X,Y);$$
the result follows by derivation at $t=0.$
\end{proof}
\begin{lem}\label{lem2 ap2}
The form $B$ on $Cl(\mathfrak{g})$ is such that
$$B(\widetilde{ad}(X),\widetilde{ad}(Y))=-\frac{1}{8\lambda}B(X,Y)$$
for all $X,Y\in\mathfrak{g},$ where $\widetilde{ad}(X)$ and $\widetilde{ad}(Y)$ represent the endomorphisms $ad(X)$ and $ad(Y)$ in $Cl(\mathfrak{g}).$ 
\end{lem}
\begin{proof}
Let us write
$$\widetilde{ad}(X)=\sum_{i<j}\gamma_{ij}(X)\ e_i\cdot e_j\hspace{.5cm}\mbox{and}\hspace{.5cm}\widetilde{ad}(Y)=\sum_{i<j}\gamma_{ij}(Y)\ e_i\cdot e_j.$$
By definition of $B$ on $Cl(\mathfrak{g})$ we have
$$B(\widetilde{ad}(X),\widetilde{ad}(Y))=\sum_{i<j}\gamma_{ij}(X)\gamma_{ij}(Y).$$
On the other hand, a direct computation yields
$$ad(X)(e_k)=\widetilde{ad}(X)\cdot e_k-e_k\cdot \widetilde{ad}(X)=2\sum_i\gamma_{ik}(X)e_i$$
where the ${\gamma_{ij}}'$s are completed such that $\gamma_{ij}=-\gamma_{ji}.$ Thus
$$ad(Y)(ad(X)(e_k))=4\sum_{i,j}\gamma_{ik}(X)\gamma_{ji}(Y)e_j$$
and
$$B(X,Y)=\lambda\ \mbox{tr}(ad(Y)\circ ad(X))=4\lambda\sum_{i,k}\gamma_{ik}(X)\gamma_{ki}(Y)=-8\lambda\sum_{i<k}\gamma_{ik}(X)\gamma_{ik}(Y);$$
the result follows.
\end{proof}
\noindent\textbf{Acknowledgments:} The author is very indebted to the referees for many comments which helped to improve considerably the writing of the paper. The author was supported by the project PAPIIT-UNAM IA106218.

\end{document}